\documentclass[oneside,english]{amsart}
\usepackage[T1]{fontenc}
\usepackage[latin9]{inputenc}
\usepackage{verbatim}
\usepackage{amsthm}
\usepackage{amssymb}

\makeatletter
\numberwithin{equation}{section}
\numberwithin{figure}{section}
\theoremstyle{plain}
\newtheorem{thm}{\protect\theoremname}[section]
  \theoremstyle{definition}
  \newtheorem{defn}[thm]{\protect\definitionname}
  \theoremstyle{plain}
  \newtheorem{lem}[thm]{\protect\lemmaname}
  \theoremstyle{remark}
  \newtheorem{rem}[thm]{\protect\remarkname}
  \theoremstyle{plain}
  \newtheorem{fact}[thm]{\protect\factname}
  \theoremstyle{plain}
  \newtheorem{cor}[thm]{\protect\corollaryname}
  \theoremstyle{plain}
  \newtheorem{conjecture}[thm]{\protect\conjecturename}
  \theoremstyle{plain}
  \newtheorem{prop}[thm]{\protect\propositionname}
  \theoremstyle{definition}
  \newtheorem{example}[thm]{\protect\examplename}
  \theoremstyle{definition}
  \newtheorem{problem}[thm]{\protect\problemname}
\newenvironment{lyxlist}[1]
{\begin{list}{}
{\settowidth{\labelwidth}{#1}
 \setlength{\leftmargin}{\labelwidth}
 \addtolength{\leftmargin}{\labelsep}
 }}
{\end{list}}
  \theoremstyle{remark}
  \newtheorem*{claim*}{\protect\claimname}

\usepackage{a4wide}
\usepackage{amssymb}
\usepackage{url}
\makeatother

\usepackage{babel}
  \providecommand{\claimname}{Claim}
  \providecommand{\conjecturename}{Conjecture}
  \providecommand{\corollaryname}{Corollary}
  \providecommand{\definitionname}{Definition}
  \providecommand{\examplename}{Example}
  \providecommand{\factname}{Fact}
  \providecommand{\lemmaname}{Lemma}
  \providecommand{\problemname}{Problem}
  \providecommand{\propositionname}{Proposition}
  \providecommand{\remarkname}{Remark}
\providecommand{\theoremname}{Theorem}

\begin{document}
\def\Ind#1#2{#1\setbox0=\hbox{$#1x$}\kern\wd0\hbox to 0pt{\hss$#1\mid$\hss}
\lower.9\ht0\hbox to 0pt{\hss$#1\smile$\hss}\kern\wd0}
\def\Notind#1#2{#1\setbox0=\hbox{$#1x$}\kern\wd0\hbox to 0pt{\mathchardef
\nn="3236\hss$#1\nn$\kern1.4\wd0\hss}\hbox to 0pt{\hss$#1\mid$\hss}\lower.9\ht0
\hbox to 0pt{\hss$#1\smile$\hss}\kern\wd0}
\def\indi{\mathop{\mathpalette\Ind{}}}
\def\nindi{\mathop{\mathpalette\Notind{}}}
\def\bdd {bdd}

\global\long\def\acl{\operatorname{acl}}

\global\long\def\dcl{\operatorname{dcl}}

\global\long\def\Avg{\operatorname{Avg}}

\global\long\def\Sk{\operatorname{Sk}}

\global\long\def\inp{\operatorname{inp}}

\global\long\def\dprk{\operatorname{dprk}}

\global\long\def\ind{\operatorname{\indi}}

\global\long\def\nind{\operatorname{\nindi}}

\global\long\def\ist{\operatorname{ist}}

\global\long\def\Aut{\operatorname{Aut}}

\global\long\def\M{\operatorname{\mathbb{M}}}

\global\long\def\NTP{\operatorname{NTP}}

\global\long\def\NIP{\operatorname{NIP}}

\global\long\def\IP{\operatorname{IP}}

\global\long\def\TP{\operatorname{TP}}

\global\long\def\tp{\operatorname{tp}}

\global\long\def\NSOP{\operatorname{NSOP}}

\global\long\def\SOP{\operatorname{SOP}}

\global\long\def\bdn{\operatorname{bdn}}

\global\long\def\vstp{\operatorname{vstp}}

\global\long\def\lstp{\operatorname{Lstp}}

\global\long\def\sf{\operatorname{sf}}

\title{Theories without the tree property of the second kind}

\author{Artem Chernikov}

\thanks{The author was supported by the Marie Curie Initial Training Network
in Mathematical Logic - MALOA - From MAthematical LOgic to Applications,
PITN-GA-2009-238381}
\begin{abstract}
We initiate a systematic study of the class of theories without the
tree property of the second kind --- $\NTP_{2}$. Most importantly,
we show: the burden is ``sub-multiplicative'' in arbitrary theories
(in particular, if a theory has $\TP_{2}$ then there is a formula
with a single variable witnessing this); $\NTP_{2}$ is equivalent
to the generalized Kim's lemma and to the boundedness of ist-weight;
the dp-rank of a type in an arbitrary theory is witnessed by mutually
indiscernible sequences of realizations of the type, after adding
some parameters --- so the dp-rank of a 1-type in any theory is always
witnessed by sequences of singletons; in $\NTP_{2}$ theories, simple
types are co-simple, characterized by the co-independence theorem,
and forking between the realizations of a simple type and arbitrary
elements satisfies full symmetry; a Henselian valued field of characteristic
$(0,0)$ is $\NTP_{2}$ (strong, of finite burden) if and only if
the residue field is $\NTP_{2}$ (the residue field and the value
group are strong, of finite burden respectively), so in particular
any ultraproduct of $p$-adics is $\NTP_{2}$; adding a generic predicate
to a geometric $\NTP_{2}$ theory preserves $\NTP_{2}$.
\end{abstract}
\maketitle

\section*{Introduction}

The aim of this paper is to initiate a systematic study of theories
without the tree property of the second kind, or $\NTP_{2}$ theories.
This class was defined by Shelah implicitly in \cite{MR1083551} in
terms of a certain cardinal invariant $\kappa_{\inp}$ (see Section
\ref{sec: kappa_inp and burden}) and explicitly in \cite{MR595012},
and it contains both simple and $\NIP$ theories. There was no active
research on the subject until the recent interest in generalizing
methods and results of stability theory to larger contexts, necessitated
for example by the developments in the model theory of important algebraic
examples such as algebraically closed valued fields \cite{HHMStableDomination}.

We give a short overview of related results in the literature. The
invariant $\kappa_{\inp}$, the upper bound for the number of independent
partitions, was considered by Tsuboi in \cite{MR805688} for the case
of stable theories. In \cite{Adl} Adler defines burden, by relativizing
$\kappa_{\inp}$ to a fixed partial type, makes the connection to
weight in simple theories and defines strong theories. Burden in the
context of $\NIP$ theories, where it is called dp-rank, was already
introduced by Shelah in \cite{Sh863} and developed further in \cite{OnshuusUsvyatsov,AlfAlexItay,KapSim}.
Results about forking and dividing in $\NTP_{2}$ theories were established
in \cite{CheKap}. In particular, it was proved that a formula forks
over a model if and only if it divides over it (see Section \ref{sec: Forking in NTP2}).
Some facts about ordered $\inp$-minimal theories and groups (that
is with $\kappa_{\inp}^{1}=1$) are proved in \cite{GoodrickDpMinGroups,PierreDpMin}.
In \cite[Theorem 4.13]{BYRandomVariables} Ben Yaacov shows that if
a structure has $\IP$, then its randomization (in the sense of continuous
logic) has $\TP_{2}$. Malliaris \cite{Maryanthe} considers $\TP_{2}$
in relation to the saturation of ultra-powers and the Keisler order.
In \cite{ZoeUnboundedPAC} Chatzidakis observes that $\omega$-free
PAC fields have $\TP_{2}$.

~

A brief description of the results in this paper. 

In Section \ref{sec: kappa_inp and burden} we introduce $\inp$-patterns,
burden, establish some of their basic properties and demonstrate that
burden is sub-multiplicative: that is, if $\bdn(a/C)<\kappa$ and
$\bdn(b/aC)<\lambda$, then $\bdn(ab/C)<\kappa\times\lambda$. As
an application we show that the value of the invariant of a theory
$\kappa_{\inp}(T)$ does not depend on the number of variables used
in the computation. This answers a question of Shelah from \cite{MR1083551}
and shows in particular that if $T$ has $\TP_{2}$, then some formula
$\phi(x,y)$ with $x$ a singleton has $\TP_{2}$. It remains open
whether burden in $\NTP_{2}$ theories is actually sub-additive.

In Section \ref{sec: NTP2 in the classification hierarchy} we describe
the place of $\NTP_{2}$ in the classification hierarchy of first-order
theories and the relationship of burden to dp-rank in NIP theories
and to weight in simple theories. We also recall some combinatorial
``structure / non-structure'' dichotomy due to Shelah, and discuss
the behavior of the $\SOP_{n}$ hierarchy restricting to $\NTP_{2}$
theories.

Section \ref{sec: Forking in NTP2} is devoted to forking (and dividing)
in $\NTP_{2}$ theories. After discussing strictly invariant types,
we give a characterization of $\NTP_{2}$ in terms of the appropriate
variants of Kim's lemma, local character and bounded weight relatively
to strict non-forking. As an application we consider theories with
dependent dividing (i.e. whenever $p\in S(N)$ divides over $M\prec N$,
there some $\phi(x,a)\in p$ dividing over $M$ and such that $\phi(x,y)$
is $\NIP$) and show that any theory with dependent dividing is $\NTP_{2}$.
Finally we observe that the the analysis from \cite{CheKap} generalizes
to a situation when one is working inside an $\NTP_{2}$ type in an
arbitrary theory.

A famous equation of Shelah ``$\NIP$ = stability + dense linear
order'' turned out to be a powerful ideological principle, at least
at the early stages of the development of $\NIP$ theories. In this
paper the equation ``$\NTP_{2}$ = simplicity + $\NIP$'' plays
an important role. In particular, it seems very natural to consider
two extremal kinds of types in $\NTP_{2}$ theories (and in general)
--- simple types and $\NIP$ types. While it is perfectly possible
for an $\NTP_{2}$ theory to have neither, they form important special
cases and are not entirely understood.

In section \ref{sec: NIP types} we look at $\NIP$ types. In particular
we show that the results of the previous section on forking localized
to a type combined with honest definitions from \cite{ExtDefI} allow
to omit the global $\NTP_{2}$ assumption in the theorem of \cite{KapSim},
thus proving that dp-rank of a type in arbitrary theory is always
witnessed by mutually indiscernible sequences of its realizations,
after adding some parameters (see Theorem \ref{thm: dp-rank is witnessed inside the type}).
We also observe that in an $\NTP_{2}$ theory, a type is $\NIP$ if
and only if every extension of it has only boundedly many global non-forking
extensions.

In Section \ref{sec: Simple types} we consider simple types (defined
as those types for which every completion satisfies the local character),
first in arbitrary theories and then in $\NTP_{2}$. While it is more
or less immediate that on the set of realizations of a simple type
forking satisfies all the properties of forking in simple theories,
the interaction between the realizations of a simple type and arbitrary
tuples seems more intricate. We establish full symmetry between realizations
of a simple type and arbitrary elements, answering a question of Casanovas
in the case of $\NTP_{2}$ theories (showing that simple types are
co-simple, see Definition \ref{def: co-simple}). Then we show that
simple types are characterized as those satisfying the co-independence
theorem and that co-simple stably embedded types are simple (so in
particular a theory is simple if and only if it is $\NTP_{2}$ and
satisfies the independence theorem).

Section \ref{sec: Examples} is devoted to examples. We give an Ax-Kochen-Ershov
type statement: a Henselian valued field of characteristic $(0,0)$
is $\NTP_{2}$ (strong, of finite burden) if and only if the residue
field is $\NTP_{2}$ (the residue field and the value group are strong,
of finite burden respectively). This is parallel to the result of
Delon for $\NIP$ \cite{Delon}, and generalizes a result of Shelah
for strong dependence \cite{Sh863}. It follows that valued fields
of Hahn series over pseudo-finite fields are $\NTP_{2}$. In particular,
every theory of an ultra-product of $p$-adics is $\NTP_{2}$ (and
in fact of finite burden). We also show that expanding a geometric
$\NTP_{2}$ theory by a generic predicate (Chatzidakis-Pillay style
\cite{ChatzidakisPillay}) preserves $\NTP_{2}$.

\subsection*{Acknowledgments}

I am grateful to Ita\"{i} Ben Yaacov, Itay Kaplan and Martin Hils
for multiple discussions around the topics of the paper. I would also
like to thank Hans Adler and Enrique Casanovas for their interest
in this work and for suggesting nice questions. Finally, I thank Ehud
Hrushovski and the referee for some corrections.

\section{Preliminaries}

As usual, we will be working in a monster model $\M$ of a complete
first-order theory $T$. We will not be distinguishing between elements
and tuples unless explicitly stated.
\begin{defn}
We will often be considering collections of sequences $\left(\bar{a}_{\alpha}\right)_{\alpha<\kappa}$
with $\bar{a}_{\alpha}=\left(a_{\alpha,i}\right)_{i<\lambda}$ (where
each $a_{\alpha,i}$ is a tuple, maybe infinite). We say that they
are \emph{mutually indiscernible} over a set $C$ if $\bar{a}_{\alpha}$
is indiscernible over $C\bar{a}_{\neq\alpha}$ for all $\alpha<\kappa$.
We will say that they are \emph{almost mutually indiscernible} over
$C$ if $\bar{a}_{\alpha}$ is indiscernible over $C\bar{a}_{<\alpha}\left(a_{\beta,0}\right)_{\beta>\alpha}$.
Sometimes we call $\left(a_{\alpha,i}\right)_{\alpha<\kappa,i<\lambda}$
an \emph{array}. We say that $\left(\bar{b}_{\alpha}\right)_{\alpha<\kappa'}$
is a \emph{sub-array} of $\left(\bar{a}_{\alpha}\right)_{\alpha<\kappa}$
if for each $\alpha<\kappa'$ there is $\beta_{\alpha}<\kappa$ such
that $\bar{b}_{\alpha}$ is a sub-sequence of $\bar{a}_{\beta_{\alpha}}$.
We say that an array is mutually indiscernible (almost mutually indiscernible)
if rows are mutually indiscernible (resp. almost mutually indiscernible).
Finally, an array is \emph{strongly indiscernible} if it is mutually
indiscernible and in addition the sequence of rows $\left(\bar{a}_{\alpha}\right)_{\alpha<\kappa}$
is an indiscernible sequence.
\end{defn}
The following lemma will be constantly used for finding indiscernible
arrays.
\begin{lem}
\label{lem: Erdos-Rado} 
\begin{enumerate}
\item For any small set $C$ and cardinal $\kappa$ there is $\lambda$
such that:\\
If $A=\left(a_{\alpha,i}\right)_{\alpha<n,i<\lambda}$ is an array,
$n<\omega$ and $\left|a_{\alpha,i}\right|\leq\kappa$, then there
is an array $B=\left(b_{\alpha,i}\right)_{\alpha<n,i<\omega}$ with
rows mutually indiscernible over $C$ and such that every finite sub-array
of $B$ has the same type over $C$ as some sub-array of $A$.
\item Let $C$ be small set and $A=\left(a_{\alpha,i}\right)_{\alpha<n,i<\omega}$
be an array with $n<\omega$. Then for any finite $\Delta\in L(C)$
and $N<\omega$ we can find $\Delta$-mutually indiscernible sequences
$\left(a_{\alpha,i_{\alpha,0}},...,a_{\alpha,i_{\alpha,N}}\right)\subset\bar{a}_{\alpha}$,
$\alpha<n$.
\end{enumerate}
\end{lem}
\begin{proof}
(1) Let $\lambda_{0}=\kappa+\left|T\right|+\left|C\right|$, $\lambda_{n+1}=\beth_{\left(2^{\lambda_{n}}\right)^{+}}$
and let $\lambda=\sum_{n<\omega}\lambda_{n}$. Now assume that we
are given an array $A=\left(a_{\alpha,i}\right)_{\alpha<n,i<\lambda}$,
and let $\bar{a}_{\alpha}=\left(a_{\alpha,i}\right)_{i<\lambda_{\alpha}}$.
By the Erd\H os-Rado theorem (see e.g. \cite[Lemma 1.2]{MR2030083})
and the choice of $\lambda_{\alpha}$'s we can find a sequence $\bar{a}_{n-1}'=\left(a_{n-1,i}'\right)_{i<\omega}$
which is indiscernible over $\bar{a}_{<n-1}$ and $ $such that every
finite subsequence of $\bar{a}_{n-1}'$ has the same type over $\bar{a}_{<n-1}$
as some finite subsequence of $\bar{a}_{n-1}$. Next, as $\left|\bar{a}_{<n-2}\bigcup\bar{a}_{n-1}'\right|\leq\lambda_{n-3}$
it follows by Erd\H os-Rado that we can find some sequence $\bar{a}_{n-2}'=\left(a_{n-2,i}'\right)_{i<\omega}$
which is indiscernible over $\bar{a}_{<n-2}\bar{a}_{n-1}'$ and such
that every finite subsequence of it has the same type over $\bar{a}_{<n-2}\bar{a}_{n-1}'$
as some subsequence of $\bar{a}_{n-2}$. Continuing in the same manner
we get sequences $\bar{a}_{n-1}',\bar{a}_{n-2}',\ldots,\bar{a}_{0}'$
and it is easy to check from the construction that they are mutually
indiscernible and give rows of an array satisfying (1).

(2) By a repeated use of the finite Ramsey theorem, see \cite[Lemma 3.5(3)]{2012arXiv1208.1341C}
for details.
\end{proof}

\begin{lem}
\label{lem: almost indiscernible array gives an indiscernible array}
Let $\left(\bar{a}_{\alpha}\right)_{\alpha<\kappa}$ be almost mutually
indiscernible over $C$. Then there are $\left(\bar{a}_{\alpha}'\right)_{\alpha<\kappa}$,
mutually indiscernible over $C$ and such that $\bar{a}_{\alpha}'\equiv_{Ca_{\alpha,0}}\bar{a}_{\alpha}$
for all $\alpha<\kappa$.\end{lem}
\begin{proof}
By Lemma \ref{lem: Erdos-Rado}, taking an automorphism, and compactness
(see \cite[Lemma 3.5(2)]{2012arXiv1208.1341C} for details).\end{proof}
\begin{defn}
\label{def: Ramsey number} Given a set of formulas $\Delta$, let
$R(\kappa,\Delta)$ be the minimal length of a sequence of singletons
sufficient for the existence of a $\Delta$-indiscernible sub-sequence
of length $\kappa$. In particular, for finite $\Delta$ we have:
\begin{enumerate}
\item $R\left(\omega,\Delta\right)=\omega$ --- by infinite Ramsey theorem,
\item $R\left(n,\Delta\right)<\omega$ for every $n<\omega$ --- by finite
Ramsey theorem,
\item $R(\kappa^{+},\Delta)\leq\beth_{\omega}\left(\kappa\right)$ for any
infinite $\kappa$ --- by Erd\H os-Rado theorem.
\end{enumerate}
\end{defn}
\begin{rem}
Let $\left(\bar{a}_{i}\right)$ be a mutually indiscernible array
over $A$. Then it is still mutually indiscernible over $\acl(A)$.\end{rem}
\begin{fact}
(see e.g. \cite{MR2800483}) Let $p\left(x\right)$ be a global type
invariant over a set $C$ (that is $\phi(x,a)\in p$ if and only if
$\phi(x,\sigma(a))\in p$ for any $\sigma\in\Aut(\M/C)$). For any
set $D\supseteq C$, and an ordinal $\alpha$, let the sequence $\bar{c}=\left\langle c_{i}\left|\, i<\alpha\right.\right\rangle $
be such that $c_{i}\models p|_{Dc_{<i}}$. Then $\bar{c}$ is indiscernible
over $D$ and its type over $D$ does not depend on the choice of
$\bar{c}$. Call this type $p^{\left(\alpha\right)}|_{D}$, and let
$p^{\left(\alpha\right)}=\bigcup_{D\supseteq C}p^{\left(\alpha\right)}|_{D}$.
Then $p^{\left(\alpha\right)}$ also does not split over $C$. 
\end{fact}
~

Finally, we assume some acquaintance with the basics of simple (e.g.
\cite{MR2585153}) and $\NIP$ (e.g. \cite{Adl}) theories.

\section{\label{sec: kappa_inp and burden} Burden and $\kappa_{\inp}$}

Let $p(x)$ be a (partial) type. 
\begin{defn}
An \emph{$\inp$-pattern} in $p(x)$ of depth $\kappa$ consists of
$\left(a_{\alpha,i}\right)_{\alpha<\kappa,i<\omega}$, $\phi_{\alpha}(x,y_{\alpha})$
and $k_{\alpha}<\omega$ such that
\begin{itemize}
\item $\left\{ \phi_{\alpha}(x,a_{\alpha,i})\right\} _{i<\omega}$ is $k_{\alpha}$-inconsistent,
for each $\alpha<\kappa$
\item $\left\{ \phi_{\alpha}(x,a_{\alpha,f(\alpha)})\right\} _{\alpha<\kappa}\cup p(x)$
is consistent, for any $f:\,\kappa\to\omega$.
\end{itemize}
The \emph{burden} of $p(x)$, denoted $\bdn(p)$, is the supremum
of the depths of all $\inp$-patterns in $p(x)$. By $\bdn(a/C)$
we mean $\bdn(tp(a/C))$.
\end{defn}
Obviously, $p(x)\subseteq q(x)$ implies $\bdn(p)\geq\bdn(q)$ and
$\bdn(p)=0$ if and only if $p$ is algebraic. Also notice that $\bdn(p)<\infty\Leftrightarrow\bdn(p)<|T|^{+}$
by compactness.

First we observe that it is sufficient to look at mutually indiscernible
$\inp$-patterns.
\begin{lem}
\label{lem: IndiscDivPattern} For $p(x)$ a (partial) type over $C$,
the following are equivalent:
\begin{enumerate}
\item There is an $\inp$-pattern of depth $\kappa$ in $p(x)$.
\item There is an array $(\bar{a}_{\alpha})_{\alpha<\kappa}$ with rows
mutually indiscernible over $C$ and $\phi_{\alpha}(x,y_{\alpha})$
for $\alpha<\kappa$ such that: 

\begin{itemize}
\item $\{\phi_{\alpha}(x,a_{\alpha,i})\}_{i<\omega}$ is inconsistent for
every $\alpha<\kappa$
\item $p(x)\cup\{\phi_{\alpha}(x,a_{\alpha,0})\}_{\alpha<\kappa}$ is consistent.
\end{itemize}
\item There is an array $\left(\bar{a}_{\alpha}\right)_{\alpha<\kappa}$
with rows almost mutually indiscernible over $C$ with the same properties.
\end{enumerate}
\end{lem}
\begin{proof}
(1)$\Rightarrow$(2) is a standard argument using Lemma \ref{lem: Erdos-Rado}
and compactness, (2)$\Rightarrow$(3) is clear and (3)$\Rightarrow$(1)
is an easy reverse induction plus compactness.
\end{proof}
We will need the following technical lemma.
\begin{lem}
\label{lem: RotateIfConsistent} Let $\left(\bar{a}_{\alpha}\right)_{\alpha<\kappa}$
be a mutually indiscernible array over $C$ and $b$ given. Let $p_{\alpha}(x,a_{\alpha,0})=\tp(b/a_{\alpha,0}C)$,
and assume that $p^{\infty}(x)=\bigcup_{\alpha<\kappa,i<\omega}p_{\alpha}(x,a_{\alpha,i})$
is consistent. Then there are $\left(\bar{a}_{\alpha}'\right)_{\alpha<\kappa}$
such that:
\begin{enumerate}
\item $\bar{a}_{\alpha}'\equiv_{a_{\alpha,0}C}\bar{a}_{\alpha}$ for all
$\alpha<\kappa$
\item $\left(\bar{a}'_{\alpha}\right)_{\alpha<\kappa}$ is a mutually indiscernible
array over $Cb$.
\end{enumerate}
\end{lem}
\begin{proof}
It is sufficient to find $b'$ such that $b'\equiv_{a_{\alpha,0}C}b$
for all $\alpha<\kappa$ and $\left(\bar{a}_{\alpha}\right)_{\alpha<\kappa}$
is mutually indiscernible over $b'C$ (then applying an automorphism
over $C$ to conclude). Let $b^{\infty}\models p^{\infty}(x)$. By
Lemma \ref{lem: Erdos-Rado}, for any finite $\Delta\in L(C)$, $S\subseteq\kappa$
and $n<\omega$, there is a $\Delta(b^{\infty})$-mutually indiscernible
sub-array $\left(a_{\alpha,i}'\right)_{\alpha\in S,i<n}$ of $\left(\bar{a}_{\alpha}\right)_{\alpha\in S}$.
Let $\sigma$ be an automorphism over $C$ sending $\left(a_{\alpha,i}'\right)_{\alpha\in S,i<n}$
to $\left(a_{\alpha,i}\right)_{\alpha\in S,i<n}$ and $b'=\sigma(b^{\infty})$.
Then $\left(a_{\alpha,i}\right)_{\alpha\in S,i<n}$ is $\Delta(b')$-mutually
indiscernible and $b'\models\bigcup_{\alpha\in S}p_{\alpha}(x,a_{\alpha,0})$,
so $b'\equiv_{a_{\alpha,0}C}b$. Conclude by compactness.
\end{proof}
~

Next lemma provides a useful equivalent way to compute the burden
of a type.
\begin{lem}
\label{lem: burden by rotation} The following are equivalent for
a partial type $p(x)$ over $C$:
\begin{enumerate}
\item There is no $\inp$-pattern of depth $\kappa$ in $p$.
\item For any $b\models p(x)$ and $\left(\bar{a}_{\alpha}\right)_{\alpha<\kappa}$,
an almost mutually indiscernible array over $C$, there is $\beta<\kappa$
and $\bar{a}'$ indiscernible over $bC$ and such that $\bar{a}'\equiv_{a_{\beta,0}C}\bar{a}_{\beta}$.
\item For any $b\models p(x)$ and $\left(\bar{a}_{\alpha}\right)_{\alpha<\kappa}$,
a mutually indiscernible array over $C$, there is $\beta<\kappa$
and $\bar{a}'$ indiscernible over $bC$ and such that $\bar{a}'\equiv_{a_{\beta,0}C}\bar{a}_{\beta}$.
\end{enumerate}
\end{lem}
\begin{proof}
(1)$\Rightarrow$(2): So let $\left(\bar{a}_{\alpha}\right)_{\alpha<\kappa}$
be almost mutually indiscernible over $C$ and $b\models p(x)$ given.
Let $p_{\alpha}(x,a_{\alpha,0})=\tp(b/a_{\alpha,0}C)$ and let $p_{\alpha}(x)=\bigcup_{i<\omega}p_{\alpha}(x,a_{\alpha,i})$.

Assume that $p_{\alpha}$ is inconsistent for each $\alpha$, by compactness
and indiscernibility of $\bar{a}_{\alpha}$ over $C$ there is some
$\phi_{\alpha}(x,a_{\alpha,0}c_{\alpha})\in p_{\alpha}(x,a_{\alpha,0})$
with $c_{\alpha}\in C$ such that $\left\{ \phi_{\alpha}(x,a_{\alpha,i}c_{\alpha})\right\} _{i<\omega}$
is $k_{\alpha}$-inconsistent. As $b\models\left\{ \phi_{\alpha}(x,a_{\alpha,0}c_{\alpha})\right\} _{\alpha<\kappa}$,
by almost indiscernibility of $\left(\bar{a}_{\alpha}\right)_{\alpha<\kappa}$
over $C$ and Lemma \ref{lem: IndiscDivPattern} we find an $\inp$-pattern
of depth $\kappa$ in $p$ -- a contradiction.

Thus $p_{\beta}(x)$ is consistent for some $\beta<\kappa$. Then
we can find $\bar{a}'$ which is indiscernible over $bC$ and such
that $\bar{a}'\equiv_{a_{\beta,0}C}\bar{a}_{\beta}$ by Lemma \ref{lem: RotateIfConsistent}.

(2)$\Rightarrow$(3) is clear.

(3)$\Rightarrow$(1): Assume that there is an $\inp$-pattern of depth
$\kappa$ in $p(x)$. By Lemma \ref{lem: IndiscDivPattern} there
is an $\inp$-pattern $\left(\bar{a}_{\alpha},\phi_{\alpha},k_{\alpha}\right)_{\alpha<\kappa}$
in $p(x)$ with $(\bar{a}_{\alpha})_{\alpha<\kappa}$ a mutually indiscernible
array over $C$. Let $b\models p(x)\cup\{\phi_{\alpha}(x,a_{\alpha,0})\}_{\alpha<\kappa}$.
On the one hand $\models\phi_{\alpha}(b,a_{\alpha,0})$, while on
the other $\left\{ \phi_{\alpha}(x,a_{\alpha,i})\right\} _{i<\omega}$
is inconsistent, thus it is impossible to find an $\bar{a}_{\alpha}'$
as required for any $\alpha<\kappa$.\end{proof}
\begin{thm}
\label{thm: product array} If there is an $\inp$-pattern of depth
$\kappa_{1}\times\kappa_{2}$ in $\tp(b_{1}b_{2}/C)$, then either
there is an $\inp$-pattern of depth $\kappa_{1}$ in $\tp(b_{1}/C)$
or there is an $\inp$-pattern of depth $\kappa_{2}$ in $\tp(b_{2}/b_{1}C)$.\end{thm}
\begin{proof}
Assume not. Without loss of generality $C=\emptyset$, and let $\left(\bar{a}_{\alpha}\right)_{\alpha\in\kappa_{1}\times\kappa_{2}}$
be a mutually indiscernible array, where we consider the product $\kappa_{1}\times\kappa_{2}$
lexicographically ordered. By induction on $\alpha<\kappa_{1}$ we
choose $\bar{a}_{\alpha}'$ and $\beta_{\alpha}\in\kappa_{2}$ such
that:
\begin{enumerate}
\item $\bar{a}_{\alpha}'$ is indiscernible over $b_{2}\bar{a}_{<\alpha}'\bar{a}_{\geq(\alpha+1,0)}$.
\item $\tp(\bar{a}_{\alpha}'/a_{(\alpha,\beta_{\alpha}),0}\bar{a}_{<\alpha}'\bar{a}_{\geq(\alpha+1,0)})=\tp(\bar{a}_{(\alpha,\beta_{\alpha})}/a_{(\alpha,\beta_{\alpha}),0}\bar{a}_{<\alpha}'\bar{a}_{\geq(\alpha+1,0)})$.
\item $\bar{a}_{\leq\alpha}'\cup\bar{a}_{\geq(\alpha+1,0)}$ is a mutually
indiscernible array.
\end{enumerate}
Assume we have managed up to $\alpha$, and we need to choose $\bar{a}_{\alpha}'$
and $\beta_{\alpha}$. Let $D=\bar{a}_{<\alpha}'\bar{a}_{\geq(\alpha+1,0)}$.
As $\left(\bar{a}_{(\alpha,\delta)}\right)_{\delta\in\kappa_{2}}$
is a mutually indiscernible array over $D$ (by assumption in the
case $\alpha=0$ and by (3) of the inductive hypothesis in the other
cases) and there is no $\inp$-pattern of depth $\kappa_{2}$ in $\tp(b_{2}/D)$,
by Lemma \ref{lem: burden by rotation}(3) there is some $\beta_{\alpha}<\kappa_{2}$
and $\bar{a}_{\alpha}'$ indiscernible over $b_{2}D$ (which gives
us (1)) such that $\tp(\bar{a}_{\alpha}'/a_{(\alpha,\beta_{\alpha}),0}D)=\tp(\bar{a}_{(\alpha,\beta_{\alpha})}/a_{(\alpha,\beta_{\alpha}),0}D)$
(which together with the inductive hypothesis gives us (2) and (3)).

So we have carried out the induction. Now it is easy to see by (1),
noticing that the first elements of $\bar{a}_{\alpha}'$ and $\bar{a}_{(\alpha,\beta_{\alpha})}$
are the same by (2), that $\left(\bar{a}_{\alpha}'\right)_{\alpha<\kappa_{1}}$
is an almost mutually indiscernible array over $b_{2}$. By Lemma
\ref{lem: almost indiscernible array gives an indiscernible array},
we may assume that in fact $\left(\bar{a}_{\alpha}'\right)_{\alpha<\kappa_{1}}$
is a mutually indiscernible array over $b_{2}$.

As there is no $\inp$-pattern of depth $\kappa_{1}$ in $\tp(b_{1}/b_{2})$,
by Lemma \ref{lem: burden by rotation} there is some $\gamma<\kappa_{1}$
and $\bar{a}$ indiscernible over $b_{1}b_{2}$ and such that $\bar{a}\equiv_{a_{\gamma,0}'}\bar{a}_{\gamma}'\equiv_{a_{(\gamma,\beta_{\gamma}),0}}\bar{a}_{(\gamma,\beta_{\gamma})}$.
As $\left(\bar{a}_{\alpha}\right)_{\alpha\in\kappa_{1}\times\kappa_{2}}$
was arbitrary, by Lemma \ref{lem: burden by rotation}(3) this implies
that there is no $\inp$-pattern of depth $\kappa_{1}\times\kappa_{2}$
in $\tp(b_{1}b_{2})$.\end{proof}
\begin{cor}
``Sub-multiplicativity'' of burden: If $\bdn(a_{i})<k_{i}$ for
$i<n$ with $k_{i}\in\omega$, then $\bdn(a_{0}...a_{n-1})<\prod_{i<n}k_{i}$.
\end{cor}
In the case of $\NIP$ theories it is known that burden is not only
sub-multiplicative, but actually sub-additive, i.e. $\bdn\left(ab\right)\leq\bdn\left(a\right)+\bdn\left(b\right)$
(by \cite{AlfAlexItay} and Fact \ref{fac: In NIP, burden =00003D dp-rank}).
Similarly, burden is sub-additive in simple theories because of the
sub-additivity of weight and Fact \ref{fac: in simple T, burden is weight}.
This motivates the following conjecture:
\begin{conjecture}
Burden is sub-additive in $\NTP_{2}$ theories.
\end{conjecture}
We also ask if burden is sub-additive in arbitrary theories.
\begin{defn}
For $n<\omega$, we let $\kappa_{\inp(T)}^{n}$ be the first cardinal
$\kappa$ such that there is no $\inp$-pattern $\left(\bar{a}_{\alpha},\phi_{\alpha}(x,y_{\alpha}),k_{\alpha}\right)$
of depth $\kappa$ with $|x|\leq n$. And let $\kappa_{\inp}(T)=\sup_{n<\omega}\kappa_{\inp}^{n}(T)$.
Notice that $\kappa_{\inp}^{m}\geq\kappa_{\inp}^{n}(T)\geq n$ for
all $n<m$, just because of having the equality in the language, and
thus $\kappa_{\inp(T)}\geq\aleph_{0}$. 
\end{defn}
We can use Theorem \ref{thm: product array} to answer a question
of Shelah \cite[Ch. III, Question 7.5]{MR1083551}.
\begin{cor}
\label{thm: k_inp^n =00003D k_inp^1}$\kappa_{\inp}(T)=\kappa_{\inp}^{n}(T)=\kappa_{\inp}^{1}(T)$,
as long as $\kappa_{\inp}^{n}$ is infinite for some $n<\omega$.
\end{cor}

\section{\label{sec: NTP2 in the classification hierarchy} $\NTP_{2}$ and
its place in the classification hierarchy}

The aim of this section is to (finally) define $\NTP_{2}$, describe
its place in the classification hierarchy of first-order theories
and what burden amounts to in the more familiar situations.
\begin{defn}
A formula $\phi(x,y)$ has $\TP_{2}$ if there is an array $\left(a_{\alpha,i}\right)_{\alpha,i<\omega}$
such that $\left\{ \phi(x,a_{\alpha,i})\right\} _{i<\omega}$ is $2$-inconsistent
for every $\alpha<\omega$ and $\left\{ \phi(x,a_{\alpha,f(\alpha)})\right\} _{\alpha<\omega}$
is consistent for any $f:\,\omega\to\omega$. Otherwise we say that
$\phi(x,y)$ is $\NTP_{2}$, and $T$ is $\NTP_{2}$ if every formula
is.\end{defn}
\begin{lem}
\label{lem: NTP_2 =00003D bounded burden} The following are equivalent
for $T$:
\begin{enumerate}
\item Every formula $\phi(x,y)$ with $|x|\leq n$ is $\NTP_{2}$.
\item $\kappa_{\inp}^{n}(T)\leq|T|^{+}$.
\item $\kappa_{\inp}^{n}(T)<\infty$. 
\item $\bdn(b/C)<\left|T\right|^{+}$ for all $b$ and $C$, with $\left|b\right|=n$.
\end{enumerate}
\end{lem}
\begin{proof}
(1)$\Rightarrow$(2): Assume we have a mutually indiscernible $\inp$-pattern
$(\bar{a}_{\alpha},\phi_{\alpha}(x,y_{\alpha}),k_{\alpha})_{\alpha<|T|^{+}}$
of depth $|T|^{+}$. By pigeon-hole we may assume that $\phi_{\alpha}(x,y_{\alpha})=\phi(x,y)$
and $k_{\alpha}=k$. Then by Ramsey and compactness we may assume
in addition that $\left(\bar{a}_{\alpha}\right)$ is a strongly indiscernible
array. If $\left\{ \phi(x,a_{\alpha,0})\land\phi(x,a_{\alpha,1})\right\} _{\alpha<n}$
is inconsistent for some $n<\omega$, then taking $b_{\alpha,i}=a_{n\alpha,i}a_{n\alpha+1,i}...a_{n\alpha+n-1,i}$,
$\left(\bigwedge_{i<n}\phi(x,y_{i}),\bar{b}_{\alpha},2\right)_{\alpha<\omega}$
is an $\inp$-pattern. Otherwise $\left\{ \phi(x,a_{\alpha,0})\land\phi(x,a_{\alpha,1})\right\} _{\alpha<\omega}$
is consistent, then taking $b_{\alpha,i}=a_{\alpha,2i}a_{\alpha,2i+1}$
we conclude that $\left(\phi(x,y_{1})\land\phi(x,y_{2}),\bar{b}_{\alpha},\left[\frac{k}{2}\right]\right)_{\alpha<\omega}$
is an $\inp$-pattern. Repeat if necessary.

The other implications are clear by compactness.\end{proof}
\begin{rem}
(1) implies (2) is from \cite{HansBurden}.
\end{rem}
It follows from the lemma and Theorem \ref{thm: k_inp^n =00003D k_inp^1}
that if $T$ has $\TP_{2}$, then some formula $\phi(x,y)$ with $|x|=1$
has $\TP_{2}$. From Lemma \ref{lem: boolean operations on inp-patterns}
it follows that if $\phi_{1}(x,y_{1})$ and $\phi_{2}(x,y_{2})$ are
$\NTP_{2}$, then $\phi_{1}(x,y_{1})\lor\phi_{2}(x,y_{2})$ is $\NTP_{2}$.
This, however, is the only Boolean operation preserving $\NTP_{2}$
(see Example \ref{ex: triangle-free random graph has TP2}).
\begin{defn}
{[}Adler{]} $T$ is called \emph{strong} if there is no $\inp$-pattern
of infinite depth in it. It is clearly a subclass of $\NTP_{2}$ theories.\end{defn}
\begin{prop}
If $\phi(x,y)$ is $\NIP$, then it is $\NTP_{2}$.\end{prop}
\begin{proof}
Let $\left(a_{\alpha,j}\right)_{\alpha,j<\omega}$ be an array witnessing
that $\phi(x,y)$ has $\TP_{2}$. But then for any $s\subseteq\omega$,
let $f(\alpha)=0$ if $\alpha\in s$, and $f(\alpha)=1$ otherwise.
Let $d\models\left\{ \phi(x,a_{\alpha,f(\alpha)})\right\} $. It follows
that $\phi(d,a_{\alpha,0})\Leftrightarrow\alpha\in s$.
\end{proof}
We recall the definition of dp-rank (e.g. \cite{AlfAlexItay}):
\begin{defn}
\label{def: dp-rank} We let the dp-rank of $p$, denoted $\dprk(p)$,
be the supremum of $\kappa$ for which there are $b\models p$ and
mutually indiscernible over $C$ (a set containing the domain of $p$)
sequences $\left(\bar{a}_{\alpha}\right)_{\alpha<\kappa}$ such that
none of them is indiscernible over $bC$.\end{defn}
\begin{fact}
The following are equivalent for a partial type $p\left(x\right)$
(by Ramsey and compactness):
\begin{enumerate}
\item $\dprk\left(p\right)>\kappa$.
\item There is an ict-pattern of depth $\kappa$ in $p\left(x\right)$,
that is $\left(\bar{a}_{i},\varphi_{i}\left(x,y_{i}\right),k_{i}\right)_{i<\kappa}$
such that $p\left(x\right)\cup\left\{ \varphi_{i}\left(x,a_{i,s\left(i\right)}\right)\right\} _{i<\kappa}\cup\left\{ \varphi_{i}\left(x,a_{i,j}\right)\right\} _{s\left(i\right)\neq j<\kappa}$
is consistent for every $s:\,\kappa\to\omega$.
\end{enumerate}
\end{fact}
It is easy to see that every $\inp$-pattern with mutually indiscernible
rows gives an ict-pattern of the same depth. On the other hand, if
$T$ is $\NIP$ then every ict-pattern gives an $\inp$-pattern of
the same depth (see \cite[Section 3]{HansBurden}). Thus we have: 
\begin{fact}
\label{fac: In NIP, burden =00003D dp-rank} 
\begin{enumerate}
\item For a partial type $p(x)$, $\bdn(p)\leq\dprk(p)$. And if $p(x)$
is an $\NIP$ type, then $\bdn(p)=\dprk(p)$
\item $T$ is strongly dependent $\Leftrightarrow$ $T$ is $\NIP$ and
strong.
\end{enumerate}
\end{fact}
\begin{prop}
If $T$ is simple, then it is $\NTP_{2}$.\end{prop}
\begin{proof}
Of course, $\inp$-pattern of the form $\left(\bar{a}_{\alpha},\phi(x,y),k\right)_{\alpha<\omega}$
witnesses the tree property.
\end{proof}
Moreover,
\begin{fact}
\cite[Proposition 8]{HansBurden} \label{fac: in simple T, burden is weight}
Let $T$ be simple. Then the burden of a partial type is the supremum
of the weights of its complete extensions. And $T$ is strong if and
only if every type has finite burden.

~\end{fact}
\begin{defn}
{[}Shelah{]} $\phi(x,y)$ is said to have $\TP_{1}$ if there are
$\left(a_{\eta}\right)_{\eta\in\omega^{<\omega}}$ and $k\in\omega$
such that:
\begin{itemize}
\item $\left\{ \phi(x,a_{\eta|i})\right\} _{i\in\omega}$ is consistent
for any $\eta\in\omega^{\omega}$
\item $\left\{ \phi(x,a_{\eta_{i}})\right\} _{i<k}$ is inconsistent for
any mutually incomparable $\eta_{0},...,\eta_{k-1}\in\omega^{<\omega}$.
\end{itemize}
\end{defn}
\begin{fact}
\label{fac: TP =00003D TP1 or TP2}\cite[III.7.7, III.7.11]{MR1083551}
Let $T$ be $\NTP_{2}$, $q(y)$ a partial type and $\phi(x,y)$ has
$\TP$ witnessed by $\left(a_{\eta}\right)_{\eta\in\omega^{<\omega}}$
with $a_{\eta}\models q$, and such that in addition $\left\{ \phi(x,a_{\eta|i})\right\} _{i\in\omega}\cup p(x)$
is consistent for any $\eta\in\omega^{\omega}$. Then some formula
$\psi(x,\bar{y})=\bigwedge_{i<k}\phi(x,y_{i})\land\chi(x)$ (where
$\chi(x)\in p(x)$) has $\TP_{1}$, witnessed by $\left(b_{\eta}\right)$
with $b_{\eta}\subseteq q(\M)$ and such that $\left\{ \phi(x,b_{\eta|i})\right\} _{i\in\omega}\cup p(x)$
is consistent.
\end{fact}
It is not stated in exactly the same form there, but immediately follows
from the proof. See \cite[Section 4]{HansBurden} and \cite[Theorem 6.6]{TreeIndiscernibility}
for a more detailed account of the argument. See \cite{KimKim_TP1}
for more details on $\NTP_{1}$.
\begin{example}
\label{ex: triangle-free random graph has TP2}Triangle-free random
graph (i.e. the model companion of the theory of graphs without triangles)
has $\TP_{2}$.\end{example}
\begin{proof}
We can find $\left(a_{ij}b_{ij}\right)_{ij<\omega}$ such that $R(a_{ij},b_{ik})$
for every $i$ and $j\neq k$, and this are the only edges around.
But then $\left\{ xRa_{ij}\land xRb_{ij}\right\} _{j<\omega}$ is
2-inconsistent for every $i$ as otherwise it would have created a
triangle, while $\left\{ xRa_{if(i)}\land xRb_{if(i)}\right\} _{i<\omega}$
is consistent for any $f:\,\omega\to\omega$. Note that the formula
$xRy$ is $\NTP_{2}$, thus demonstrating that a conjunction of two
$\NTP_{2}$ formulas need not be $\NTP_{2}$.
\end{proof}
A similar argument shows that the theory of a $K_{n}$-free random
graph has $\TP_{2}$ for all $n\geq3$. In fact it is known that the
triangle-free random graph is rosy and 2-dependent (in the sense of
\cite{Sh886}), thus there is no implication between rosiness and
$\NTP_{2}$, and between k-dependence and $\NTP_{2}$ for $k>1$.

\subsection{On the $\SOP_{n}$ hierarchy restricted to $\NTP_{2}$ theories}

We recall the definition of $\SOP_{n}$ for $n\geq2$ from \cite[Definition 2.5]{MR1402297}:
\begin{defn}

\begin{enumerate}
\item Let $n\geq3$. We say that a formula $\phi\left(x,y\right)$ has $\SOP_{n}$
if there are $\left(a_{i}\right)_{i\in\omega}$ such that:

\begin{enumerate}
\item There is an infinite chain: $\models\phi\left(a_{i},a_{j}\right)$
for all $i<j<\omega$,
\item There are no cycles of length $n$: $\models\neg\exists x_{0}\ldots x_{n-1}\bigwedge_{j=i+1\left(\mod n\right)}\phi\left(x_{i},x_{j}\right)$.
\end{enumerate}
\item $\phi\left(x,y\right)$ has $\SOP_{2}$ if and only if it has $\TP_{1}$.
\item For a theory $T$, $\SOP\Rightarrow\ldots\Rightarrow\SOP_{n+1}\Rightarrow\SOP_{n}\Rightarrow\ldots\Rightarrow\SOP_{3}\Rightarrow\SOP_{2}\Rightarrow\TP$.
\item By Fact \ref{fac: TP =00003D TP1 or TP2} we see that restricting
to $\NTP_{2}$ theories, the last 2 items coincide.
\end{enumerate}
\end{defn}
The following are the standard examples showing that the $\SOP_{n}$
hierarchy is strict for $n\geq3$:
\begin{example}
\cite[Claim 2.8]{MR1402297}
\begin{enumerate}
\item For $n\geq3$, let $T_{n}$ be the model completion of the theory
of directed graphs (no self-loops or multiple edges) with no directed
cycles of length $\leq n$. Then it has $\SOP_{n}$ but not $\SOP_{n+1}$.
\item For odd $n\geq3$, the model completion of the theory of graphs with
no odd cycles of length $\leq n$, has $\SOP_{n}$ but not $\SOP_{n+1}$.
\item Consider the model companion of a theory in the language $\left(<_{n,l}\right)_{l\leq n}$
saying:

\begin{enumerate}
\item $x<_{n,m-1}y\rightarrow x<_{n,m}y$,
\item $x<_{n,n}y$,
\item $\neg\left(x<_{n,n-1}x\right)$,
\item if $l+k+1=m\leq n$ then $x<_{n,l}y\,\land\, y<_{n,k}z\,\rightarrow\, x<_{n,m}z$.
\end{enumerate}

It eliminates quantifiers.

\end{enumerate}
\end{example}
However, all these examples have $\TP_{2}$.
\begin{proof}
(1) Let $\phi\left(x,y_{1}y_{2}\right)=xRy_{1}\land y_{2}Rx$. For
$i\in\omega$ we choose sequencese $\left(a_{i,j}b_{i,j}\right)_{j\in\omega}$
such that $\models R\left(a_{i,j},b_{i,k}\right)$ and $R\left(b_{i,j},a_{i,k}\right)$
for all $j<k\in\omega$, and these are the only edges around --- it
is possible as no directed cycles are created. Now for any $i$, if
there is $c\models\phi\left(x,a_{i,0}b_{i,0}\right)\land\phi\left(x,a_{i,1}b_{i,1}\right)$,
then we would have a directed cycle $c,b_{i,0},a_{i,1}$ of length
$3$ --- a contradiction. On the other hand, given and $i_{0}<\ldots<i_{n}$
and $j_{0},\ldots,j_{n}$ there has to be an element $a\models\bigwedge_{\alpha\leq n}\phi\left(x,a_{i_{\alpha},j_{\alpha}}b_{i_{\alpha},j_{\alpha}}\right)$
as there are no directed cycles created. Thus $\phi\left(x,y_{1}y_{2}\right)$
has $\TP_{2}$.

(2) and (3) Similar.
\end{proof}
This naturally leads to the following question:
\begin{problem}
Is the $\SOP_{n}$ hierarchy strict restricting to $\NTP_{2}$ theories?
\end{problem}
In \cite[Exercise III.7.12]{MR1083551} Shelah suggests an example
of a theory satisfying $\NTP_{2}+\NSOP$ which is not simple. However,
his example doesn't seem to work.

\section{\label{sec: Forking in NTP2} Forking in $\NTP_{2}$}

In \cite[Theorem 2.4]{MR1833481} Kim gives several equivalents to
the simplicity of a theory in terms of the behavior of forking and
dividing. 
\begin{fact}
The following are equivalent:
\begin{enumerate}
\item $T$ is simple.
\item $\phi(x,a)$ divides over $A$ if and only if $\left\{ \phi(x,a_{i})\right\} _{i<\omega}$
is inconsistent for every Morley sequence $\left(a_{i}\right)_{i<\omega}$
over $A$. 
\item Dividing in $T$ satisfies local character.
\end{enumerate}
\end{fact}
In this section we show an analogous characterization of $\NTP_{2}$.
But first we recall some facts about forking and dividing in $\NTP_{2}$
theories and introduce some terminology.
\begin{defn}
\label{def: strict invariance, etc}
\begin{enumerate}
\item A type $p(x)\in S(C)$ is \emph{strictly invariant} over $A$ if it
is Lascar invariant over $A$ and for any small $B\subseteq C$ and
$a\models p|_{B}$, we have that $\tp(B/aA)$ does not divide over
$A$ (we can replace ``does not divide'' by ``does not fork''
$C=\M$). For example, a definable type or a global type which is
both an heir and a coheir over $M$, are strictly invariant over $M$. 
\item We will write $a\ind_{c}^{\ist}b$ when $\tp(a/bc)$ can be extended
to a global type $p(x)$ strictly invariant over $A$. 
\item We say that $\left(a_{i}\right)_{<\omega}$ is a strict Morley sequence
over $A$ if it is indiscernible over $A$ and $a_{i}\ind_{A}^{\ist}a_{<i}$
for all $i<\omega$.
\item As usual, we will write $a\ind_{c}^{u}b$ if $\tp(a/bc)$ is finitely
satisfiable in $c$, $a\ind_{c}^{d}b$ ($a\ind_{c}^{f}b$) if $\tp(a/bc)$
does not divide (resp. does not fork) over $c$.
\item We write $a\ind_{c}^{i}b$ if $\tp(a/bc)$ can be extended to a global
type $p(x)$ Lascar invariant over $c$. We point out that if $a\ind_{c}^{i}b$
and $\left(b_{i}\right)_{i<\omega}$ is a $c$-indiscernible sequence
with $b_{0}=b$, then it is actually indiscernible over $a$.
\item If $T$ is simple, then $\ind^{i}=\ind^{\ist}$. And if $T$ is $\NIP$,
then $\ind^{i}=\ind^{f}$.
\item We say that a set $A$ is an \emph{extension base} if every type over
$A$ has a global non-forking extension. Every model is an extension
base (because every type has a global coheir). A theory in which every
set is an extension base is called extensible.
\end{enumerate}
\end{defn}
Strictly invariant types exist in any theory (but it is not true that
every type over a model has a global extension which is strictly invariant
over the same model). In fact, there are theories in which over any
set there is some type without a global strictly invariant extension
(see \cite{NFSpectra}).
\begin{lem}
\label{lem: finding strictly invariant type} Let $p(x)$ be a global
type invariant over $A$, and let $M\supset A$ be $|A|^{+}$-saturated.
Then $p$ is strictly invariant over $M$.\end{lem}
\begin{proof}
It is enough to show that $p$ is an heir over $M$. Let $\phi(x,c)\in p$.
By saturation of $M$, $\tp(c/A)$ is realized by some $c'\in M$.
But as $p$ is invariant over $A$, $\phi(x,c')\in p$ as wanted.
\end{proof}
One of the main uses of strict invariance is the following criterion
for making indiscernible sequences mutually indiscernible without
changing their type over the first elements.
\begin{lem}
\label{lem: mutual indiscernibility criteria} Let $\left(\bar{a}_{i}\right)_{i<\kappa}$and
$C$ be given, with $\bar{a}_{i}$ indiscernible over $C$ and starting
with $a_{i}$. If $a_{i}\ind_{C}^{\ist}a_{<i}$, then there are mutually
$C$-indiscernible $\left(\bar{b}_{i}\right)_{i<\kappa}$ such that
$\bar{b}_{i}\equiv_{a_{i}C}\bar{a}_{i}$.\end{lem}
\begin{proof}
Enough to show for finite $\kappa$ by compactness. So assume we have
chosen $\bar{a}_{0}',...,\bar{a}_{n-1}'$, and lets choose $\bar{a}_{n}'$.
As $a_{n}\ind_{C}^{\ist}a_{<n}$, there are $\bar{a}_{0}''...\bar{a}_{n-1}''\equiv_{Ca_{0}...a_{n-1}}\bar{a}_{0}'...\bar{a}_{n-1}'$
and such that $a_{n}\ind_{C}^{\ist}\bar{a}_{<n}''$. As $a_{n}\ind_{C\bar{a}''_{<n,\neq j}}^{i}\bar{a}''_{j}$
for $j<n$, it follows by the inductive assumption and Definition
\ref{def: strict invariance, etc}(5) that $\bar{a}''_{j}$ is indiscernible
over $a_{n}\bar{a}''_{\neq j}$. On the other hand $\bar{a}_{0}''...\bar{a}_{n-1}''\ind_{C}^{f}a_{n}$,
and so by basic properties of forking there is some $\bar{a}'_{n}\equiv_{Ca_{n}}\bar{a}_{n}$
indiscernible over $\bar{a}_{0}'',...,\bar{a}_{n-1}''$. Conclude
by Lemma \ref{lem: almost indiscernible array gives an indiscernible array}.\end{proof}
\begin{rem}
This argument is essentially from \cite[Section 5]{ShelahDependentCont}.
\end{rem}
We recall a result about forking and dividing in $\NTP_{2}$ theories
from \cite{CheKap}.
\begin{fact}
\cite{CheKap}\label{Fac: Forking in an NTP2 theory} Let $T$ be
$\NTP_{2}$ and $M\models T$.
\begin{enumerate}
\item Every $p\in S(M)$ has a global strictly invariant extension.
\item For any $a$, $\phi(x,a)$ divides over $M$ if and only if $\phi(x,a)$
forks over $M$, if and only if for every $\left(a_{i}\right)_{i<\omega}$,
a strict Morley sequence in $\tp(a/M)$, $\left\{ \phi(x,a_{i})\right\} _{i<\omega}$
is inconsistent.
\item In fact, just assuming that $A$ is an extension base, we still have
that $\phi(x,a)$ does not divide over $A$ if and only if $\phi(x,a)$
does not fork over $A$.
\end{enumerate}
\end{fact}

\subsection{Characterization of $\NTP_{2}$}

Now we can give a method for computing the burden of a type in terms
of dividing with each member of an $\ind^{\ist}$-independent sequence.
\begin{lem}
\label{lem: burden and ind-ist} Let $p(x)$ be a partial type over
$C$. The following are equivalent:
\begin{enumerate}
\item There is an $\inp$-pattern of depth $\kappa$ in $p(x)$.
\item There is $d\models p(x)$, $D\supseteq C$ and $\left(a_{\alpha}\right)_{\alpha<\kappa}$
such that $a_{\alpha}\ind_{D}^{\ist}a_{<\alpha}$ and $d\nind_{D}^{d}a_{\alpha}$
for all $\alpha<\kappa$. 
\end{enumerate}
\end{lem}
\begin{proof}
(1)$\Rightarrow$(2): Let $\left(\bar{a}_{\alpha},\phi_{\alpha}(x,y_{\alpha}),k_{\alpha}\right)_{\alpha<\kappa}$
be an $\inp$-pattern in $p(x)$ with $(\bar{a}_{\alpha})$ mutually
indiscernible over $C$. Let $q_{\alpha}(\bar{y}_{\alpha})$ be a
non-algebraic type finitely satisfiable in $\bar{a}_{\alpha}$ and
extending $\tp\left(a_{\alpha0}/C\right)$. Let $M\supseteq C\left(\bar{a}_{\alpha}\right)_{\alpha<\kappa}$
be $\left(|C|+\kappa\right)^{+}$-saturated. Then $q_{\alpha}$ is
strictly invariant over $M$ by Lemma \ref{lem: finding strictly invariant type}.
For $\alpha,i<\kappa$ let $b_{\alpha,i}\models q_{\alpha}\restriction_{M\left(b_{\alpha,j}\right)_{\alpha<\kappa,j<i}\left(b_{\beta,i}\right)_{\beta<\alpha}}$.
Let $e_{\alpha}=b_{\alpha,\alpha}$. Now we have:
\begin{itemize}
\item $e_{\alpha}\ind_{M}^{\ist}e_{<\alpha}$: as $e_{\alpha}\models q_{\alpha}\restriction_{e_{<\alpha}M}$.
\item there is $d\models p(x)\cup\left\{ \phi_{\alpha}(x,e_{\alpha})\right\} _{\alpha<\kappa}$:
it is easy to see by construction that for any $\Delta\in L(C)$ and
$\alpha_{0}<...<\alpha_{n-1}<\kappa$, if $\models\Delta(e_{\alpha_{0}},...,e_{\alpha_{n-1}})$,
then $\models\Delta(a_{\alpha_{0},i_{0}},...,a_{\alpha_{n-1},i_{n-1}})$
for some $i_{0},...,i_{n-1}<\omega$. By assumption on $\left(\bar{a}_{\alpha}\right)_{\alpha<\kappa}$
and compactness it follows that $p(x)\cup\left\{ \phi_{\alpha}(x,e_{\alpha})\right\} _{\alpha<\kappa}$
is consistent.
\item $\phi_{\alpha}(x,e_{\alpha})$ divides over $M$: notice that $\left(b_{\alpha,\alpha+i}\right)_{i<\omega}$
is an $M$-indiscernible sequence starting with $e_{\alpha}$, as
$b_{\alpha,\alpha+i}\models q_{\alpha}\restriction_{M\left(b_{\alpha,\alpha+j}\right)_{j<i}}$
and $q_{\alpha}$ is finitely satisfiable in $M$. As $\tp(\bar{b}_{\alpha})$
is finitely satisfiable in $\bar{a}_{\alpha}$, we conclude that $\left\{ \phi_{\alpha}(x,b_{\alpha,\alpha+i})\right\} _{i<\omega}$
is $k_{\alpha}$-inconsistent.
\end{itemize}
(2)$\Rightarrow$(1): Let $d\models p(x)$, $D\supseteq C$ and $\left(a_{\alpha}\right)_{\alpha<\kappa}$
such that $a_{\alpha}\ind_{D}^{\ist}a_{<\alpha}$ and $d\nind_{D}^{f}a_{\alpha}$
for all $\alpha<\kappa$ be given. Let $\phi_{\alpha}(x,a_{\alpha})\in\tp(d/a_{\alpha}D)$
be a formula dividing over $D$, and let $\bar{a}_{\alpha}$ indiscernible
over $D$ and starting with $a_{\alpha}$ witness it. By Lemma \ref{lem: IndiscDivPattern}
we can find a $\left(\bar{a}_{\alpha}'\right)_{\alpha<\kappa}$, mutually
indiscernible over $D$ and such that $\bar{a}_{\alpha}'\equiv_{a_{\alpha}D}\bar{a}_{\alpha}$.
It follows that $\left\{ \phi_{\alpha}(x,y_{\alpha}),\bar{a}_{\alpha}'\right\} _{\alpha<\kappa}$
is an $\inp$-pattern of depth $\kappa$ in $p(x)$.\end{proof}
\begin{defn}
We say that dividing satisfies \emph{generic local character} if for
every $A\subseteq B$ and $p(x)\in S(B)$ there is some $A'\subseteq B$
with $\left|A'\right|\leq\left|T\right|^{+}$ and such that: for any
$\phi(x,b)\in p$, if $b\ind_{A}^{\ist}A'$, then $\phi(x,b)$ does
not divide over $AA'$.
\end{defn}
Of course, the local character of dividing implies the generic local
character. We are ready to prove the main theorem of this section.
\begin{thm}
\label{thm: NTP2 iff Kims lemma iff bounded weight} The following
are equivalent:
\begin{enumerate}
\item $T$ is $\NTP_{2}$.
\item $T$ has absolutely bounded $\ind^{\ist}$-weight: for every $M$,
$b$ and $\left(a_{i}\right)_{i<|T|^{+}}$ with $a_{i}\ind_{M}^{\ist}a_{<i}$,
$b\ind_{M}^{d}a_{i}$ for some $i<|T|^{+}$.
\item $T$ has bounded $\ind^{\ist}$-weight: for every $M$ there is some
$\kappa_{M}$ such that given $b$ and $\left(a_{i}\right)_{i<\kappa_{M}}$
with $a_{i}\ind_{M}^{\ist}a_{<i}$, $b\ind_{M}^{d}a_{i}$ for some
$i<\kappa_{M}$. 
\item $T$ satisfies ``Kim's lemma'': for any $M\models T$, $\phi(x,a)$
divides over $M$ if and only if $\left\{ \phi(x,a_{i})\right\} _{i<\omega}$
is inconsistent for every strict Morley sequence over $M$.
\item Dividing in $T$ satisfies generic local character.
\end{enumerate}
\end{thm}
\begin{proof}
(1) implies (2): Assume that there are $M$, $b$ and $\left(a_{i}\right)_{i<|T|^{+}}$
with $a_{i}\ind_{M}^{\ist}a_{<i}$ and $b\nind_{M}^{d}a_{i}$ for
all $i$. But then by Lemma \ref{lem: burden and ind-ist} $\bdn(b/M)\geq\left|T\right|^{+}$,
thus $T$ has $\TP_{2}$ by Lemma \ref{lem: NTP_2 =00003D bounded burden}.

(2) implies (3) is clear.

(1) implies (4): by Fact \ref{Fac: Forking in an NTP2 theory}(1)+(2).

(4) implies (3): assume that we have $M$, $b$ and $\left(a_{i}\right)_{i<\kappa}$
such that, letting $\kappa=\beth_{\left(2^{|M|}\right)^{+}}$, $a_{i}\ind_{M}^{\ist}a_{<i}$
and $b\nind_{M}^{d}a_{i}$ for all $i<\kappa$. We may assume that
dividing is always witnessed by the same formula $\phi(x,y)$. Extracting
an $M$-indiscernible sequence $\left(a_{i}'\right)_{i<\omega}$ from
$\left(a_{i}\right)_{i<\kappa}$ by Erd\"{o}s-Rado, we get a contradiction
to (4) as $\left\{ \phi(x,a_{i}')\right\} _{i<\omega}$ is still consistent,
$\left(a_{i}'\right)$ is a strict Morley sequence over $M$ and $\phi(x,a_{0}')$
divides over $M$.

(3) implies (1): Assume that $\varphi\left(x,y\right)$ has $\TP_{2}$,
let $A=\left(\bar{a}_{\alpha}\right)_{\alpha<\omega}$ with $\bar{a}_{\alpha}=\left(a_{\alpha i}\right)_{i<\omega}$
be a strongly indiscernible array witnessing it (so rows are mutually
indiscernible and the sequence of rows is indiscernible). Let $M\supset A$
be some $\left|A\right|^{+}$-saturated model, and assume that $\kappa_{M}$
is as required by (3). Let $\lambda=\beth_{\left(2^{\left|M\right|}\right)^{+}}$
and $\mu=\left(2^{2^{\lambda}}\right)^{+}$. Adding new elements and
rows by compactness, extend our strongly indiscernible array to one
of the form $\left(\bar{a}_{\alpha}\right)_{\alpha\in\omega+\mu^{*}}$
with $\bar{a}_{\alpha}=\left(a_{\alpha i}\right)_{i\in\lambda}$.
By all the indiscernibility around it follows that $\bar{a}_{\alpha}\ind_{A}^{u}\bar{a}_{<\alpha}$
for all $\alpha<\mu$. As there can be at most $2^{2^{\lambda}}$
global types from $S_{\lambda}\left(\M\right)$ that are finitely
satisfiable in $A$, without loss of generality there is some $p\left(\bar{x}\right)\in S_{\lambda}\left(\M\right)$
finitely satisfiable in $A$ and such that $\bar{a}_{\alpha}\models p\left(\bar{x}\right)|_{A\bar{a}_{<\alpha}}$.

By Lemma \ref{lem: finding strictly invariant type}, $p\left(\bar{x}\right)$
is strictly invariant over $M$. We choose $\left(\bar{b}_{\alpha}\right)_{\alpha<\kappa_{M}}$
such that $\bar{b}_{\alpha}\models p|_{M\bar{b}_{<\alpha}}$. 

By the choice of $\lambda$ and Erd\"{o}s-Rado, for each $\alpha<\kappa_{M}$
there is $i_{\alpha}<\lambda$ and $\bar{d}_{\alpha}$ such that $\bar{d}_{\alpha}$
is an $M$-indiscernible sequence starting with $b_{\alpha i_{\alpha}}$
and such that type of every finite subsequence of it is realized by
some subsequence of $\bar{b}_{\alpha}$. Now we have:
\begin{itemize}
\item $d_{\alpha0}\ind_{M}^{\ist}d_{<\alpha0}$ (as $d_{\alpha0}=b_{\alpha i_{\alpha}}$
and $\bar{b}_{\alpha}\ind_{M}^{\ist}\bar{b}_{<\alpha}$),
\item $\varphi\left(x,d_{\alpha0}\right)$ divides over $M$ (as $\bar{d}_{\alpha}$
is $M$-indiscernible and $\left\{ \varphi\left(x,d_{\alpha i}\right)\right\} _{i\in\omega}$
is inconsistent by construction),
\item $\left\{ \varphi\left(x,d_{\alpha0}\right)\right\} _{\alpha<\kappa_{M}}$
is consistent (follows by construction).
\end{itemize}
Taking some $c\models\left\{ \varphi\left(x,d_{\alpha0}\right)\right\} _{\alpha<\kappa_{M}}$
we get a contradiction to (3).

(5) implies (2): Let $p(x)=\tp(b/B)$ with $B=M\cup\bigcup_{i<\left|T\right|^{+}}a_{i}$.
Letting $A=M$, it follows by generic local character that there is
some $A'\subseteq B$ with $\left|A'\right|\leq\left|T\right|$, such
that $b\ind_{MA'}^{d}a$ for any $a\in B$ with $a\ind_{M}A'$. Let
$i\in\left|T\right|$ be such that $i>\left\{ j\,:\, a_{j}\in A'\right\} $.
Then $a_{i}\ind_{M}^{\ist}A$, but also $b\nind_{MA'}^{d}a_{i}$ (by
left transitivity as $A'\ind_{M}^{d}a_{i}$ and $b\nind_{M}^{d}a_{i}$)
--- a contradiction.

(1) implies (5): Let $p\left(x\right)\in S\left(B\right)$ and $A\subseteq B$
be given. By induction on $i<\left|T\right|^{+}$ we try to choose
$a_{i}\in B$ and $\varphi_{i}\left(x,a_{i}\right)\in p$ such that
$a_{i}\ind_{A}^{\ist}a_{<i}$ and $\varphi_{i}\left(x,a_{i}\right)$
divides over $a_{<i}A$. But then by Lemma \ref{lem: burden and ind-ist}
$\bdn(b/A)\geq\left|T\right|^{+}$, thus $T$ has $\TP_{2}$ by Lemma
\ref{lem: NTP_2 =00003D bounded burden}. So we had to get stuck,
and letting $A'=\bigcup a_{i}$ witnesses the generic local character.\end{proof}
\begin{rem}

\begin{enumerate}
\item The proof of the equivalences shows that in (2) and (3) we may replace
$a\ind_{C}^{\ist}b$ by ``$\tp(a/bC)$ extends to a global type which
is both an heir and a coheir over $C$''.
\item From the proof one immediately gets a similar characterization of
strongness. Namely, the following are equivalent:

\begin{enumerate}
\item $T$ is strong.
\item For every $M$, finite (or even singleton) $b$ and $\left(a_{i}\right)_{i<\omega}$
with $a_{i}\ind_{M}^{\ist}a_{<i}$, $b\ind_{M}^{d}a_{i}$ for some
$i<\omega$.
\item For every $A\subseteq B$ and $p(x)\in S(B)$ there is some \emph{finite}
$A'\subseteq B$ such that: for any $\phi(x,b)\in p$, if $b\ind_{A}^{\ist}A'$,
then $\phi(x,b)$ does not divide over $AA'$.
\end{enumerate}
\end{enumerate}
\end{rem}
If we are working over a somewhat saturated model and consider only
small sets, then we actually have the generic local character with
respect to $\ind^{u}$ in the place of $\ind^{\ist}$.
\begin{lem}
\label{lem: almost mutually indiscernible for ind^u} Let $\left(\bar{a}_{i}\right)_{i<\kappa}$and
$C$ be given, $\bar{a}_{i}$ starting with $a_{i}$. If $\bar{a}_{i}$
is indiscernible over $\bar{a}_{<i}C$ and $a_{i}\ind_{C}^{i}a_{<i}$,
then $\left(\bar{a}_{i}\right)_{i<\kappa}$ is almost mutually indiscernible
over $C$.\end{lem}
\begin{prop}
Let $T$ be $\NTP_{2}$. Let $M$ be $\kappa$-saturated, $p(x)\in S(M)$
and $A\subset M$ of size $<\kappa$. Then there is $A\subseteq A'\subset M$
of size $<\kappa$ such that for any $\phi(x,a)\in p$, if $a\ind_{A}^{i}A'$
then $\phi(x,a)$ does not fork over $A'$.\end{prop}
\begin{proof}
Assume not, then we can choose inductively on $\alpha<|T|^{+}$:
\begin{enumerate}
\item $\bar{a}_{\alpha}\subseteq M$ such that $a_{\alpha,0}\ind_{A}^{i}A_{\alpha}$
and $\bar{a}_{\alpha}$ is $A_{\alpha}$-indiscernible, $A_{\alpha}=A\cup\bigcup_{\beta<\alpha}\bar{a}_{\beta}$.
\item $\phi_{\alpha}(x,y_{\alpha})$ such that $\phi_{\alpha}(x,a_{\alpha,0})\in p$
and $\left\{ \phi_{\alpha}(x,a_{\alpha,i})\right\} _{i<\omega}$ is
inconsistent.
\end{enumerate}
(1) is possible by saturation of $M$. But then by Lemma \ref{lem: almost mutually indiscernible for ind^u},
$\left(\bar{a}_{\alpha}\right)_{\alpha<|T|^{+}}$ are almost mutually
indiscernible.
\end{proof}

\subsection{Dependent dividing}
\begin{defn}
We say that $T$ has \emph{dependent dividing} if given $M\preceq N$
and $p(x)\in S(N)$ dividing over $M$, then there is a dependent
formula $\phi(x,y)$ and $c\in N$ such that $\phi(x,c)\in p$ and
$\phi(x,c)$ divides over $M$.\end{defn}
\begin{prop}

\begin{enumerate}
\item If $T$ has dependent dividing, then it is $\NTP_{2}$.
\item If $T$ has simple dividing, then it is simple.
\end{enumerate}
\end{prop}
\begin{proof}

\begin{enumerate}
\item In fact we will only use that dividing is always witnessed by an instance
of an $\NTP_{2}$ formula. Assume that $T$ has $\TP_{2}$ and let
$\phi(x,y)$ witness this. Let $T_{\Sk}$ be a Skolemization of $T$,
$\phi(x,y)$ still has $\TP_{2}$ in $T_{\Sk}$. Then as in the proof
of Theorem \ref{thm: NTP2 iff Kims lemma iff bounded weight}, for
any $\kappa$ we can find $\left(b_{i}\right)_{i<\kappa}$, $a$ and
$M$ such that $a\models\left\{ \phi(x,b_{i})\right\} _{i<\kappa}$,
$\phi(x,b_{i})$ divides over $M$ and $\tp\left(b_{i}/b_{<i}M\right)$
has a global heir-coheir over $M$, all in the sense of $T_{\Sk}$.
Taking $M_{i}=\mbox{Sk}(Mb_{i})\models T$, and now working in $T$,
we still have that $a\nind_{M}^{d}M_{i}$ and $M_{i}\ind_{M}^{\ist}M_{<i}$
(as $\tp(M_{i}/M_{<i}M)$ still has a global heir-coheir over $M$).
But then for each $i$ we find some $d_{i}\in M_{i}$ and $\NTP_{2}$
formulas $\phi_{i}(x,y_{i})\in L$ such that $a\models\left\{ \phi_{i}(x,d_{i})\right\} $
and $\phi_{i}(x,d_{i})$ divides over $M$, witnessed by $\bar{d}_{i}$
starting with $d_{i}$. We may assume that $\phi_{i}=\phi'$, and
this contradicts $\phi'$ being $\NTP_{2}$.
\item Similar argument shows that if $T$ has simple dividing, then it is
simple.
\end{enumerate}
\end{proof}
Of course, if $T$ is $\NIP$, then it has dependent dividing, and
for simple theories it is equivalent to the stable forking conjecture.
It is natural to ask if every $\NTP_{2}$ theory $T$ has dependent
dividing.

\subsection{Forking and dividing inside an $\NTP_{2}$ type}
\begin{defn}
A partial type $p(x)$ over $C$ is said to be $\NTP_{2}$ if the
following does not exist: $\left(\bar{a}_{\alpha}\right)_{\alpha<\omega}$,
$\phi(x,y)$ and $k<\omega$ such that $\left\{ \phi(x,a_{\alpha i})\right\} _{i<\omega}$
is $k$-inconsistent for every $\alpha<\omega$ and $\left\{ \phi(x,a_{\alpha f(\alpha)})\right\} _{\alpha<\omega}\cup p(x)$
is consistent for every $f:\,\omega\to\omega$. Of course, $T$ is
$\NTP_{2}$ if and only if every partial type is $\NTP_{2}$. Also
notice that if $p(x)$ is $\NTP_{2}$, then every extension of it
is $\NTP_{2}$ and that $q((x_{i})_{i<\kappa})=\bigcup_{i<\kappa}p(x_{i})$
is $\NTP_{2}$ (follows from Theorem \ref{thm: product array}). 
\end{defn}
For the later use we will need a generalization of the results from
\cite{CheKap} working inside a partial $\NTP_{2}$ type, and with
no assumption on the theory. 
\begin{lem}
Let $p(x)$ be an $\NTP_{2}$ type over $M$. Assume that $p(x)\cup\left\{ \phi(x,a)\right\} $
divides over $M$, then there is a global coheir $q(x)$ extending
$\tp(a/M)$ such that $p(x)\cup\left\{ \phi(x,a_{i})\right\} _{i<\omega}$
is inconsistent for any sequence $\left(a_{i}\right)_{i<\omega}$
with $a_{i}\models q|_{a_{<i}M}$.\end{lem}
\begin{proof}
The proof of \cite[Lemma 3.12]{CheKap} goes through.
\end{proof}

\begin{lem}
Assume that $\tp(a_{i}/C)=p(x)$ for all $i$ and that $\tp(a_{i}/a_{<i}C)$
has a strictly invariant extension to $p(\M)\cup C$. Then there are
mutually $C$-indiscernible $\left(\bar{b}_{i}\right)_{i<\kappa}$
such that $\bar{b}_{i}\equiv_{a_{i}C}\bar{a}_{i}$.\end{lem}
\begin{proof}
The assumption is sufficient for the proof of Lemma \ref{lem: mutual indiscernibility criteria}
to work.
\end{proof}

\begin{lem}
\label{lem: inside NTP2 type, strict Morley sequence is universal}
Let $p(x)$ over $M$ be $\NTP_{2}$, $a\in p(\M)$, $c\in M$ and
assume that $p(x)\cup\left\{ \phi(x,ac)\right\} $ divides over $M$.
Assume that $\tp(a/M)$ has a strictly invariant extension $p'(y)\in S(p(\M))$.
Then for any $\left(a_{i}\right)_{i<\omega}$ such that $a_{i}\models p'|_{a_{<i}M}$,
$p(x)\cup\left\{ \phi(x,a_{i}c)\right\} _{i<\omega}$ is inconsistent.\end{lem}
\begin{proof}
Let $\left(\bar{a}_{0}c\right)$ with $a_{0,0}=a_{0}$ be an $M$-indiscernible
sequence witnessing that $p(x)\cup\left\{ \phi(x,a_{0}c)\right\} $
divides over $M$. Let $\bar{a}_{i}$ be its image under an $M$-automorphism
sending $a_{0}$ to $a_{i}$. By Lemma \ref{lem: mutual indiscernibility criteria}(2)
we can find $\left(\bar{b}_{i}\right)_{i<\omega}$ mutually indiscernible
over $M$ and with $\bar{b}_{i}\equiv_{a_{i}M}\bar{a}_{i}$. By the
choice of $\bar{b}_{i}$'s and compactness, there is some $\psi(x)\in p(x)$
such that $\left\{ \psi(x)\land\phi(x,b_{i,j}c)\right\} _{j<\omega}$
is $k$-inconsistent for all $i<\omega$. It follows that $p(x)\cup\left\{ \phi(x,a_{i}c)\right\} _{i<\omega}$
is inconsistent as $p$ is $\NTP_{2}$.
\end{proof}
We need a version of the Broom lemma localized to an $\NTP_{2}$ type.
\begin{lem}
\label{lem: local broom lemma} Let $p(x)$ be an $\NTP_{2}$ type
over $M$ and $p'(x)$ be a partial global type invariant over $M$.
Suppose that $p(x)\cup p'(x)\vdash\bigvee_{i<n}\phi_{i}(x,c)$ and
each $\phi_{i}(x,c)$ divides over $M$. Then $p(x)\cup p'(x)$ is
inconsistent.\end{lem}
\begin{proof}
Follows from the proof of \cite[Lemma 3.1]{CheKap}.\end{proof}
\begin{cor}
\label{lem: inside NTP2 type, strict invariant extensions exist}
Let $p(x)$ be an $\NTP_{2}$ type over $M$ and $a\in p(\M)$. Then
$\tp(a/M)$ has a strictly invariant extension $p'(x)\in S(p(\M)\cup M)$.\end{cor}
\begin{proof}
Following the proof of \cite[Proposition 3.7]{CheKap} but using Lemma
\ref{lem: local broom lemma} in place of the Broom lemma.
\end{proof}

And finally,
\begin{prop}
\label{prop: forking=00003Ddividing inside NTP2 type} Let $p(x)$
be an $\NTP_{2}$ type over $M$, $a\in p(\M)\cup M$ and assume that
$\left\{ \phi(x,a)\right\} \cup p(x)$ does not divide over $M$.
Then there is $p'(x)\in S(p(\M)\cup M)$ which does not divide over
$M$ and $\left\{ \phi(x,a)\right\} \cup p(x)\subset p'(x)$.\end{prop}
\begin{proof}
By compactness, it is enough to show that if $p(x)\cup\left\{ \phi(x,ac)\right\} \vdash\bigvee_{i<n}\phi_{i}(x,a_{i}c_{i})$
with $a,a_{i}\in p(\M)$ and $c,c_{i}\in M$, then $p(x)\cup\left\{ \phi_{i}(x,a_{i}c_{i})\right\} $
does not divide over $M$ for some $i<n$. As in the proof of \cite[Corollary 3.16]{CheKap},
let $\left(a^{j}a_{0}^{j}...a_{n-1}^{j}\right)_{j<\omega}$ be a strict
Morley sequence in $\tp(aa_{0}...a_{n-1}/M)$, which exists by Lemma
\ref{lem: inside NTP2 type, strict invariant extensions exist}. Notice
that $\left(a^{j}ca_{0}^{j}c_{0}...a_{n-1}^{j}c_{n-1}\right)_{j<\omega}$
is still indiscernible over $M$. Then $p(x)\cup\left\{ \phi(x,a^{j}c)\right\} _{j<\omega}$
is consistent, which implies that $p(x)\cup\left\{ \phi_{i}(x,a_{i}^{j}c_{i})\right\} _{j<\omega}$
is consistent for some $i<n$. But then by Lemma \ref{lem: inside NTP2 type, strict Morley sequence is universal},
$p(x)\cup\left\{ \phi_{i}(x,a_{i}c_{i})\right\} $ does not divide
over $M$ --- as wanted.
\end{proof}

\section{\label{sec: NIP types} $\NIP$ types}

Let $T$ be an arbitrary theory.
\begin{defn}

\begin{enumerate}
\item A partial type $p(x)$ over $C$ is called $\NIP$ if there is no
$\phi(x,y)\in L$, $\left(a_{i}\right)_{i\in\omega}$ with $a_{i}\models p(x)$
and $\left(b_{s}\right)_{s\subseteq\omega}$ such that $\models\phi(a_{i},b_{s})$
$\Leftrightarrow$ $i\in s$.
\item The roles of $a$'s and $b$'s in the definition are interchangeable.
It is easy to see that any extension of an $\NIP$ type is again $\NIP$,
and that the type of several realizations of an $\NIP$ type is again
$\NIP$.
\item $p(x)$ is $\NIP$ $\Leftrightarrow$ $\dprk(p)<|T|^{+}$ $\Leftrightarrow$
$\dprk(p)<\infty$ (see Definition \ref{def: dp-rank}).
\end{enumerate}
\end{defn}
\begin{lem}
\label{lem: Everything works inside NIP type} Let $p(x)$ be an $\NIP$
type.
\begin{enumerate}
\item Let $\bar{a}=\left(a_{\alpha}\right)_{\alpha<\kappa}$ be an indiscernible
sequence over $A$ with $a_{\alpha}$ from $p(\M)$, and $c$ be arbitrary.
If $\kappa=\left(|a_{\alpha}|+|c|\right)^{+}$, then some non-empty
end segment of $\bar{a}$ is indiscernible over $Ac$.
\item Let $\left(\bar{a}_{\alpha}\right)_{\alpha<\kappa}$ be mutually indiscernible
(over $\emptyset$), with $\bar{a}_{\alpha}=\left(a_{\alpha i}\right)_{i<\lambda}$
from $p(\M)$. Assume that $\bar{a}=\left(a_{0i}a_{1i}...\right)_{i<\lambda}$
is indiscernible over $A$. Then $\left(\bar{a}_{\alpha}\right)_{\alpha<\kappa}$
is mutually indiscernible over $A$.
\end{enumerate}
\end{lem}
Standard proofs of the corresponding results for $\NIP$ theories
go through, see e.g. \cite{Adl}.

\subsection{Dp-rank of a type is always witnessed by an array of its realizations}

In \cite{KapSim} Kaplan and Simon demonstrate that inside an $\NTP_{2}$
theory, dp-rank of a type can always be witnessed by mutually indiscernible
sequences of realizations of the type. In this section we show that
the assumption that the theory is $\NTP_{2}$ can be omitted, thus
proving the following general theorem with no assumption on the theory.
\begin{thm}
\label{thm: dp-rank is witnessed inside the type} Let $p(x)$ be
an $\NIP$ partial type over $C$, and assume that $\dprk(p)\geq\kappa$.
Then there is $C'\supseteq C$, $b\models p(x)$ and $\left(\bar{a}_{\alpha}\right)_{\alpha<\kappa}$
with $\bar{a}_{\alpha}=\left(a_{\alpha i}\right)_{i<\omega}$ such
that:
\begin{itemize}
\item $a_{\alpha i}\models p(x)$ for all $\alpha,i$
\item $\left(\bar{a}_{\alpha}\right)_{\alpha<\kappa}$ are mutually indiscernible
over $C'$
\item None of $\bar{a}_{\alpha}$ is indiscernible over $bC'$.
\item $\left|C'\right|\leq\left|C\right|+\kappa$.
\end{itemize}
\end{thm}
\begin{cor}
It follows that dp-rank of a 1-type is always witnessed by mutually
indiscernible sequences of singletons.
\end{cor}
We will use the following result from \cite[Proposition 1.1]{ExtDefI}:
\begin{fact}
\label{fac: honest definition} Let $p(x)$ be a (partial) $\NIP$
type, $A\subseteq p(\M)$ and $\phi(x,c)$ given. Then there is $\theta(x,d)$
with $d\in p(\M)$ such that:
\begin{enumerate}
\item $\theta(A,d)=\phi(A,c)$,
\item $\theta(x,d)\cup p(x)\rightarrow\phi(x,c)$.
\end{enumerate}
\end{fact}
~

We begin by showing that the burden of a dependent type can always
be witnessed by mutually indiscernible sequences from the set of its
realizations.
\begin{lem}
\label{lem: burden witnessed inside} Let $p(x)$ be a dependent partial
type over $C$ of burden $\geq\kappa$. Then we can find $\left(\bar{d}_{\alpha}\right)_{\alpha<\kappa}$
witnessing it, mutually indiscernible over $C$ and with $\bar{d}_{i}\subseteq p(\M)\cup C$.\end{lem}
\begin{proof}
Let $\lambda$ be large enough compared to $\left|C\right|$. Assume
that $\mbox{bdn}(p)\geq\kappa$, then by compactness we can find $\left(\bar{b}_{\alpha},\phi_{\alpha}(x,y_{\alpha}),k_{\alpha}\right)_{i<n}$
such that $\bar{b}_{\alpha}=\left(b_{\alpha i}\right)_{i<\lambda}$,
$\left\{ \phi_{\alpha}(x,b_{\alpha i})\right\} _{\alpha<\kappa}$
is $k_{\alpha}$-inconsistent and $p(x)\cup\left\{ \phi_{\alpha}(x,b_{\alpha f(\alpha)})\right\} _{i<n}$
is consistent for every $f:\,\kappa\to\lambda$, let $a_{f}$ realize
it. Set $A=\left\{ a_{f}\right\} _{f\in\lambda^{\kappa}}\subseteq p(\M)$.

By Fact \ref{fac: honest definition}, let $\theta_{\alpha i}(x,d_{\alpha i})$
be an honest definition of $\phi_{\alpha}(x,b_{\alpha i})$ over $A$
(with respect to $p(x)$), with $d_{\alpha i}\in p(\mathbb{M})$.
As $\lambda$ is very large, we may assume that $\theta_{\alpha i}=\theta_{\alpha}$. 

Now, as $\theta_{\alpha}(x,d_{\alpha i})\cup p(x)\rightarrow\phi_{\alpha}(x,b_{\alpha i})$,
it follows that there is some $\psi_{\alpha}(x,c)\in p$ such that
letting $\chi_{\alpha}(x,y_{1}y_{2})=\theta_{\alpha}(x,y_{1})\land\psi_{\alpha}(x,y_{2})$,
$\left\{ \chi(x,d_{\alpha i}c_{\alpha})\right\} _{i<\omega}$ is $k_{\alpha}$-inconsistent. 

On the other hand, $\left\{ \chi_{\alpha}(x,d_{\alpha f(\alpha)}c_{\alpha})\right\} _{\alpha<\kappa}\cup p(x)$
is consistent, as the corresponding $a_{f}$ realizes it. Thus this
array still witnesses that burden of $p$ is at least $\kappa$.
\end{proof}
We will also need the following lemma.
\begin{lem}
\label{lem: non-dividing gives inv extension in NIP type} Let $p(x)$
be an $\NIP$ type over $M\models T$ 
\begin{enumerate}
\item Assume that $a\in p(\M)\cup M$ and $\phi(x,a)$ does not divide over
$M$, then there is a type $q(x)\in S(p(\M)\cup M)$ invariant under
$M$-automorphisms and with $\phi(x,a)\in q$.
\item Let $p'(x)\supset p(x)$ be an $M$ invariant type such that $p^{(\omega)}$
is an heir-coheir over $M$. If $\left(a_{i}\right)_{i<\omega}$ is
a Morley sequence in $p'$ and indiscernible over $bM$ with $b\in p(\M)$,
then $\tp(b/MI)$ has an $M$-invariant extension in $S(p(\M)\cup M)$.
\end{enumerate}
\end{lem}
\begin{proof}
(1) As $\NIP$ type is in particular an $\NTP_{2}$ type, by Lemma
\ref{prop: forking=00003Ddividing inside NTP2 type} we find a type
$q(x)\in S(p(\M))$ which doesn't divide over $M$ and such that $\phi(x,a)\in q$.
It is enough to show that $q(x)$ is Lascar-invariant over $M$. Assume
that we have an $M$-indiscernible sequence $(a_{i})_{i<\omega}$in
$p(\M)$ such that $\phi(x,a_{0})\land\neg\phi(x,a_{1})\in q$. But
then $\left\{ \phi(x,a_{2i})\land\phi(x,a_{2i+1})\right\} _{i<\omega}$
is inconsistent, so $q$ divides over $M$ --- a contradiction. Easy
induction shows the same for $a_{0}$ and $a_{1}$ at Lascar distance
$n$.

(2) By Lemma \ref{lem: inside NTP2 type, strict Morley sequence is universal}
and (1).
\end{proof}
Now for the \emph{proof of Theorem \ref{thm: dp-rank is witnessed inside the type}}.
The point is that first the array witnessing dp-rank of our type $p(x)$
can be dragged inside the set of realizations of $p$ by Lemma \ref{lem: burden witnessed inside}.
Then, combined with the use of Proposition \ref{lem: non-dividing gives inv extension in NIP type}
instead of the unrelativized version, the proof of Kaplan and Simon
\cite[Section 3.2]{KapSim} goes through working inside $p(\M)$. 
\begin{problem}
\label{Prob: burden witnessed inside} Is the analogue of Lemma \ref{lem: burden witnessed inside}
true for the burden of an arbitrary type in an $\NTP_{2}$ theory?
\end{problem}
We include some partial observations to justify it.
\begin{prop}
The answer to the Problem \ref{Prob: burden witnessed inside} is
positive in the following cases:
\begin{enumerate}
\item $T$ satisfies dependent forking (so in particular if $T$ is $\NIP$).
\item $T$ is simple.
\end{enumerate}
\end{prop}
\begin{proof}
(1): Recall that if $\mbox{bdn}(p)\geq\kappa$, then we can find $\left(b_{i}\right)_{i<\kappa}$,
$a\models p$ and $M\supseteq C$ such that $a\nind_{M}^{d}b_{i}$
and $b_{i}\ind_{M}^{\ist}b_{<i}$. Notice that $p(x)$ still has the
same burden in the sense of a Skolemization $T^{\Sk}$. Choose inductively
$M_{i}\supseteq M\cup b_{i}$ such that $M_{i}\ind_{M}^{\ist}b_{<i}$,
let $M_{i}=Sk(M\cup b_{i})$. Let $\phi(x,b_{i})$ witness this dividing
with $\phi(x,y)$ an $\NIP$ formula, we can make $\bar{b}_{i}$ mutually
indiscernible. Now the proof of Lemma \ref{lem: burden witnessed inside}
goes through.

(2): Let $p(x)\in S(A)$, $a\models p(x)$ and let $(b_{i})_{i<\kappa}$
independent over $A$, with $a\nind_{A}b_{i}$. Without loss of generality
$A=\emptyset$. Consider $\mbox{tp}(a/b_{0})$ and take $I=(a_{i})_{i<|T|^{+}}$
such that $a\overset{\frown}{}I$ is a Morley sequence in it. By extension
and automorphism we may assume $b_{>0}\ind_{ab_{0}}I$, together with
$a\ind_{b_{0}}I$ implies $b_{>0}\ind_{b_{0}}I$, thus $b_{>0}\ind I$
(as $b_{>0}\ind b_{0}$).

Assume that $I$ is a Morley sequence over $\emptyset$, then by simplicity
$a_{i}\ind b_{0}$ for some $i$, contradicting $a_{i}\equiv_{b_{0}}a$
and $a\nind b_{0}$. Thus by indiscernibility $a\nind a_{<n}$ for
some $n$, while $\{a_{<n}\}\cup b_{>0}$ is an independent set. 

Repeating this argument inductively and using the fact that the burden
of a type in a simple theory is the supremum of the weights of its
completions (Fact \ref{fac: in simple T, burden is weight}) allows
to conclude.
\end{proof}

\subsection{NIP types inside an $\NTP_{2}$ theory}

We give a characterization of $\NIP$ types in $\NTP_{2}$ theories
in terms of the number of non-forking extensions of its completions.
\begin{thm}
Let $T$ be $\NTP_{2}$, and let $p(x)$ be a partial type over $C$.
The following are equivalent:
\begin{enumerate}
\item $p$ is $\NIP$.
\item Every $p'\supseteq p$ has boundedly many global non-forking extensions.
\end{enumerate}
\end{thm}
\begin{proof}
(1)$\Rightarrow$(2): A usual argument shows that a non-forking extension
of an $\NIP$ type is in fact Lascar-invariant (see Lemma \ref{lem: non-dividing gives inv extension in NIP type}),
thus there are only boundedly many such.

(2)$\Rightarrow$(1): Assume that $p(x)$ is not $\NIP$, that is
there are $I=\left(b_{i}\right)_{i\in\omega}$ such that such that
for any $s\subseteq\omega$, $p_{s}(x)=p(x)\cup\left\{ \phi(x,b_{i})\right\} _{i\in s}\cup\left\{ \neg\phi(x,b_{i})\right\} _{i\notin s}$
is consistent. Let $q(y)$ be a global non-algebraic type finitely
satisfiable in $I$. Let $M\supseteq IC$ be some $|IC|^{+}$-saturated
model. It follows that $q^{(\omega)}$ is a global heir-coheir over
$M$ by Lemma \ref{lem: finding strictly invariant type}. Take an
arbitrary cardinal $\kappa$, and let $J=\left(c_{i}\right)_{i\in\kappa}$
be a Morley sequence in $q$ over $M$. We claim that for any $s\subseteq\kappa$,
$p_{s}(x)$ does not divide over $M$. First notice that $p_{s}(x)$
is consistent for any $s$, as $\mbox{tp}(J/M)$ is finitely satisfiable
in $I$. But as for any $k<\omega$, $(c_{ki}c_{ki+1}...c_{k(i+1)-1})_{i<\omega}$
is a Morley sequence in $q^{(k)}$, together with Fact \ref{Fac: Forking in an NTP2 theory}
this implies that $p_{s}(x)|_{c_{0}...c_{k-1}}$ does not divide over
$M$ for any $k<\omega$, thus by indiscernibility of $J$, $p_{s}(x)$
does not divide over $M$, thus has a global non-forking extension
by Fact \ref{Fac: Forking in an NTP2 theory}. 

As there are only boundedly many types over $M$, there is some $p'\in S(M)$
extending $p$, with unboundedly many global non-forking extensions.\end{proof}
\begin{rem}
(2)$\Rightarrow$(1) is just a localized variant of an argument from
\cite{NFSpectra}.
\end{rem}

\section{\label{sec: Simple types}Simple types}

\subsection{Simple and co-simple types}

Simple types, to the best of our knowledge, were first defined in
\cite[§4]{KimPillayHart} in the form of (2).
\begin{defn}
\label{def: simple} We say that a partial type $p(x)\in S(A)$ is
\emph{simple} if it satisfies any of the following equivalent conditions:
\begin{enumerate}
\item There is no $\phi(x,y)$, $\left(a_{\eta}\right)_{\eta\in\omega^{<\omega}}$
and $k<\omega$ such that: $\left\{ \phi(x,a_{\eta i})\right\} _{i<\omega}$
is $k$-inconsistent for every $\eta\in\omega^{<\omega}$ and $\left\{ \phi(x,a_{\eta\restriction i})\right\} _{i<\omega}\cup p(x)$
is consistent for every $\eta\in\omega^{\omega}$. 
\item Local character: If $B\supseteq A$ and $p(x)\subseteq q(x)\in S(B)$,
then $q(x)$ does not divide over $AB'$ for some $B'\subseteq B$,
$|B'|\leq|T|$.
\item Kim's lemma: If $\left\{ \phi(x,b)\right\} \cup p(x)$ divides over
$B\supseteq A$ and $(b_{i})_{i<\omega}$ is a Morley sequence in
$\mbox{tp}(b/B)$, then $p(x)\cup\{\phi(x,b_{i})\}_{i<\omega}$ is
inconsistent.
\item Bounded weight: Let $B\supseteq A$ and $\kappa\geq\beth_{\left(2^{|B|}\right)^{+}}$.
If $a\models p(x)$ and $(b_{i})_{i<\kappa}$ is such that $b_{i}\ind_{B}^{f}b_{<i}$,
then $a\ind_{B}^{d}b_{i}$ for some $i<\kappa$. 
\item For any $B\supseteq A$, if $b\ind_{B}^{f}a$ and $a\models p(x)$,
then $a\ind_{B}^{d}b$.
\end{enumerate}
\end{defn}
\begin{proof}
~
\begin{lyxlist}{00.00.0000}
\item [{(1)$\Rightarrow$(2):}] Assume (2) fails, then we choose $\phi_{\alpha}(x,b_{\alpha})\in q(x)$
$k_{\alpha}$-dividing over $A\cup B_{\alpha}$, with $B_{\alpha}=\{b_{\beta}\}_{\beta<\alpha}\subseteq B$,
$|B_{\alpha}|\leq|\alpha|$ by induction on $\alpha<|T|^{+}$. Then
w.l.o.g. $\phi_{\alpha}=\phi$ and $k_{\alpha}=k$. Now construct
a tree in the usual manner, such that $\{\phi(x,a_{\eta i})\}_{i<\omega}$
is inconsistent for any $\eta\in\omega^{<\omega}$ and $\{\phi(x,a_{\eta|i})\}_{i<\omega}\cup p(x)$
is consistent for any $\eta\in\omega^{\omega}$.
\item [{(2)$\Rightarrow$(3):}] Let $I=(|T|^{+})^{*}$, and $(b_{i})_{i\in I}$
be Morley over $B$ in $\mbox{tp}(b/B)$. Assume that $a\models p(x)\cup\{\phi(x,b_{i})\}_{i\in I}$.
By (2), $\mbox{tp}(a/(b_{i})_{i\in I}B)$ does not divide over $B(b_{i})_{i\in I_{0}}$
for some $I_{0}\subseteq I$, $|I_{0}|\leq|T|$. Let $i_{0}\in I$,
$i_{0}<I_{0}$. Then $(b_{i})_{i\in I_{0}}\ind_{B}^{f}b_{i_{0}}$,
and thus $\phi(x,b_{i_{0}})$ divides over $BI_{0}$ - a contradiction.
\item [{(3)$\Rightarrow$(4):}] Assume not, then by Erd\"{o}s-Rado and
finite character find a Morley sequence over $B$ and a formula $\phi(x,y)$
such that $\models\phi(a,b_{i})$ and $\phi(x,b_{i})$ divides over
$B$, contradiction to (3).
\item [{(4)$\Rightarrow$(5):}] For $\kappa$ as in (4), let $I=(b_{i})_{i<\kappa}$
be a Morley sequence over $B$, indiscernible over $Ba$ and with
$b_{0}=b$. By (4), $a\ind_{B}^{d}b_{i}$ for some $i<\kappa$, and
so $a\ind_{B}^{d}b$ by indiscernibility.
\item [{(5)$\Rightarrow$(1):}] Let $(b_{\eta})_{\eta\in\omega^{<\omega}}$
witness the tree property of $\phi(x,y)$, such that $\{\phi(x,b_{\eta|i})\}_{i<\omega}\cup p(x)$
is consistent for every $\eta\in\omega^{\omega}$. Then by Ramsey
and compactness we can find $(b_{i})_{i\leq\omega}$ indiscernible
over $a$, $\models\phi(a,b_{i})$ and $\phi(x,b_{i})$ divides over
$b_{<i}A$. Taking $B=A\cup\{b_{i}\}_{i<\omega}$ we see that $a\nind_{B}^{d}b_{\omega}$,
while $b_{\omega}\ind_{B}^{f}a$ (as it is finitely satisfiable in
$B$ by indiscernibility) - a contradiction to (5).
\end{lyxlist}
\end{proof}
\begin{rem}
\label{rem: tuples are simple} Let $p(x)\in S(A)$ be simple.
\begin{enumerate}
\item Any $q(x)\supseteq p(x)$ is simple.
\item Let $p(x)\in S(A)$ be simple and $C\subseteq p(\mathbb{M})$. Then
$\mbox{tp}(C/A)$ is simple. 
\end{enumerate}
\end{rem}
\begin{proof}
(1): Clear, for example by (1) from the definition.

(2): Let $C=(c_{i})_{i\leq n}$, and we show that for any $B\supseteq A$,
if $b\ind_{B}^{f}C$, then $C\ind_{B}^{d}b$ by induction on the size
of $C$. Notice that $b\ind_{Bc_{<n}}^{f}c_{n}$ and $c_{n}\models p$,
thus $c_{n}\ind_{Bc_{<n}}^{d}b$. By the inductive assumption $c_{<n}\ind_{B}^{d}b$,
thus $c_{\leq n}\ind_{B}^{d}b$.
\end{proof}
We give a characterization in terms of local ranks.
\begin{prop}
The following are equivalent:
\begin{enumerate}
\item $p(x)$ is simple in the sense of Definition \ref{def: simple}.
\item $D(p,\Delta,k)<\omega$ for any finite $\Delta$ and $k<\omega$. 
\end{enumerate}
\end{prop}
\begin{proof}
Standard proof goes through.\end{proof}
\begin{lem}
Let $p(x)\in S(A)$ be simple, $a\models p(x)$ and $B\supseteq A$
arbitrary. Then $a\ind_{B_{0}}^{f}B$ for some $|B_{0}|\leq|T|^{+}$.\end{lem}
\begin{proof}
Standard proof using ranks goes through.
\end{proof}
It follows that in the Definition \ref{def: simple} we can replace
everywhere ``dividing'' by ``forking''.
\begin{lem}
\label{lem: Forking=00003Ddividing over an extension base} Let $p(x)\in S(A)$
be simple. If $A$ is an extension base, then $\{\phi(x,c)\}\cup p(x)$
forks over $A$ if and only if it divides over $A$. \end{lem}
\begin{proof}
Assume that $\{\phi(x,c)\}\cup p(x)$ does not divide over $A$, but
$\{\phi(x,c)\}\cup p(x)\vdash\bigvee_{i<n}\phi_{i}(x,c_{i})$ and
each of $\phi_{i}(x,c_{i})$ divides over $A$. As $A$ is an extension
base, let $(c_{i}c_{0,i}...c_{n-1,i})$ be a Morley sequence in $\mbox{tp}(cc_{0}...c_{n-1}/A)$.
As $p(x)\cup\{\phi(x,c)\}$ does not divide over $A$, let $a\models p(x)\cup\{\phi(x,c_{i})\}$,
but then $p(x)\cup\{\phi_{i}(x,c_{i,j})\}_{j<\omega}$ is consistent
for some $i<n$, contradicting Kim's lemma. \end{proof}
\begin{problem}
Let $q(x)$ be a non-forking extension of a complete type $p(x)$,
and assume that $q(x)$ is simple. Does it imply that $p(x)$ is simple?
\end{problem}
~

Unlike stability or $\mbox{NIP}$, it is possible that $\phi(x,y)$
does not have the tree property, while $\phi^{*}(x',y')=\phi(y',x')$
does. This forces us to define a dual concept.
\begin{defn}
\label{def: co-simple} A partial type $p(x)$ over $A$ is \emph{co-simple}
if it satisfies any of the following equivalent properties:
\begin{enumerate}
\item No formula $\phi(x,y)\in L(A)$ has the tree property witnessed by
some $(a_{\eta})_{\eta\in\omega^{<\omega}}$ with $a_{\eta}\subseteq p(\mathbb{M})$.
\item Every type $q(x)\in S(BA)$ with $B\subseteq p(\mathbb{M})$ does
not divide over $AB'$ for some $B'\subseteq B$, $|B'|\leq\left(|A|+|T|\right)^{+}$.
\item Let $(a_{i})_{i<\omega}\subseteq p(\mathbb{M})$ be a Morley sequence
over $BA$, $B\subseteq p(\mathbb{M})$ and $\phi(x,y)\in L(A)$.
If $\phi(x,a_{0})$ divides over $BA$ then $\{\phi(x,a_{i})\}_{i<\omega}$
is inconsistent.
\item Let $B\subseteq p(\mathbb{M})$ and $\kappa\geq\beth_{\left(2^{|B|+|A|}\right)^{+}}$.
If $(b_{i})_{i<\kappa}\subseteq p(\mathbb{M})$ is such that $b_{i}\ind_{AB}^{f}b_{<i}$
and $a$ arbitrary, then $a\ind_{AB}^{d}b_{i}$ for some $i<\kappa$. 
\item For $B\subseteq p(\mathbb{M})$, if $a\models p$ and $a\ind_{AB}^{f}b$,
then $b\ind_{AB}^{d}a$.
\end{enumerate}
\end{defn}
\begin{proof}
Similar to the proof in Definition \ref{def: simple}.\end{proof}
\begin{rem}
It follows that if $p(x)$ is a co-simple type over $A$ and $B\subseteq p(\M)$,
then any $q(x)\in S(AB)$ extending $p$ is co-simple (while adding
the parameters from outside of the set of solutions of $p$ may ruin
co-simplicity).
\end{rem}
~

It is easy to see that $T$ is simple $\Leftrightarrow$ every type
is simple $\Leftrightarrow$ every type is co-simple. What is the
relation between simple and co-simple in general?
\begin{example}
\label{ex: co-simple not simple type} There is a co-simple type over
a model which is not simple.\end{example}
\begin{proof}
Let $T$ be the theory of an infinite triangle-free random graph,
this theory eliminates quantifiers. Let $M\models T$, $m\in M$ and
consider $p(x)=\{xRm\}\cup\{\neg xRa\}_{a\in M\setminus\{m\}}$ -
a non-algebraic type over $M$. As there can be no triangles, if $a,b\models p(x)$
then $\neg aRb$. It follows that for any $A\subseteq p(\mathbb{M})$
and any $B$, $B\nind_{M}^{d}A$ $\Leftrightarrow$ $B\cap A\neq\emptyset$.
So $p(x)$ is co-simple, for example by checking the bounded weight
(Definition \ref{def: co-simple}(4)).

For each $\alpha<\omega$, take $\left(b_{\alpha,i}'b_{\alpha,i}''\right)_{i<\omega}$
such that $b_{\alpha,i}'Rb_{\alpha,j}''$ for all $i\neq j$, and
no other edges between them or to elements of $M$. Then $\left\{ xRb_{\alpha,i}'\land xRb_{\alpha,i}''\right\} _{i<\omega}$
is $2$-inconsistent for every $\alpha$, while $p(x)\cup\left\{ xRb_{\alpha,\eta(\alpha)}'\land xRb_{\alpha,\eta(\alpha)}''\right\} _{\alpha<\omega}$
is consistent for every $\eta:\,\omega\to\omega$. Thus $p(x)$ is
not simple by Definition \ref{def: simple}(1).
\end{proof}
However, this $T$ has $\mbox{TP}_{2}$ by Example \ref{ex: triangle-free random graph has TP2}.
\begin{problem}
Is there a simple, non co-simple type in an arbitrary theory?
\end{problem}

\subsection{Simple types are co-simple in $\mbox{NTP}_{2}$ theories}

In this section we assume that $T$ is $\NTP_{2}$ (although some
lemmas remain true without this restriction). In particular, we will
write $\ind$ to denote non-forking/non-dividing when working over
an extension base as they are the same by Fact \ref{Fac: Forking in an NTP2 theory}(3).
\begin{lem}
\label{lem: Chain condition} Weak chain condition: Let $A$ be an
extension base, $p(x)\in S(A)$ simple. Assume that $a\models p(x)$,
$I=(b_{i})_{i<\omega}$ is a Morley sequence over $A$ and $a\ind_{A}b_{0}$.
Then there is an $aA$-indiscernible $J\equiv_{Ab_{0}}I$ satisfying
$a\ind_{A}J$ .\end{lem}
\begin{proof}
Let $a\models\phi(x,b_{0})$, then $\{\phi(x,b_{0})\}\cup p(x)$ does
not divide over $A$.
\begin{claim*}
$\{\phi(x,b_{0})\land\phi(x,b_{1})\}\cup p(x)$ does not divide over
$A$.\end{claim*}
\begin{proof}
As $p(x)$ satisfies Definition \ref{def: simple}(3), $(b_{2i}b_{2i+1})_{i<\omega}$
is a Morley sequence over $A$ and $\{\phi(x,b_{i})\}_{i<\omega}\cup p(x)$
is consistent.
\end{proof}

By iterating the claim and compactness, we conclude that $\bigcup_{i<\omega}p(x,b_{i})$
does not divide over $A$, where $p(x,b_{0})=\mbox{tp}(a/b_{0})$.
As $A$ is an extension base and forking equals dividing, there is
$a'\models\bigcup_{i<\omega}p(x,b_{i})$ satisfying $a'\ind_{A}I$.
By Ramsey, compactness and the fact that $a'b_{i}\equiv_{A}ab_{0}$
we find a sequence as wanted.\end{proof}
\begin{rem}
If fact, in \cite{CheBY} we demonstrate that in an $\NTP_{2}$ theory
this lemma holds over extension bases with $I$ just an indiscernible
sequence, not necessarily Morley.\end{rem}
\begin{lem}
\label{lem: Making almost mutually indiscernible} Let $A$ be an
extension base, $p\in S(A)$ simple. For $i<\omega$, Let $\bar{a}_{i}$
be a Morley sequence in $p(x)$ over $A$ starting with $a_{i}$,
and assume that $(a_{i})_{i<\omega}$ is a Morley sequence in $p(x)$.
Then we can find $\bar{b}_{i}\equiv_{Aa_{i}}\bar{a}_{i}$ such that
$(\bar{b}_{i})_{i<\omega}$ are mutually indiscernible over $A$.\end{lem}
\begin{proof}
W.l.o.g. $A=\emptyset$. 

First observe that by simplicity of $p$, $\{a_{i}\}_{i<\omega}$
is an independent set.

For $i<\omega$, we choose inductively $\bar{b}_{i}$ such that:
\begin{enumerate}
\item $\bar{b}_{i}\equiv_{a_{i}}\bar{a}_{i}$
\item $\bar{b}_{i}$ is indiscernible over $a_{>i}\bar{b}_{<i}$
\item $a_{>i+1}\bar{b}_{\leq i}\ind a_{i+1}$
\item $a_{\geq i+1}\ind\bar{b}_{\leq i}$
\end{enumerate}

Base step: As $a_{>0}\ind a_{0}$ and $\mbox{tp}(a_{>0})$ is simple
by Remark \ref{rem: tuples are simple} and Lemma \ref{lem: Chain condition},
we find an $a_{>0}$-indiscernible $\bar{b}_{0}\equiv_{a_{0}}\bar{a}_{0}$
with $a_{>0}\ind\bar{b}_{0}$.

Induction step: Assume that we have constructed $\bar{b}_{0},...,\bar{b}_{i-1}$.
By (3) for $i-1$ it follows that $a_{>i}\bar{b}_{<i}\ind a_{i}$.
Again by Remark \ref{rem: tuples are simple} and Lemma \ref{lem: Chain condition}
we find an $a_{>i}\bar{b}_{<i}$-indiscernible sequence $\bar{b}_{i}\equiv_{a_{i}}\bar{a}_{i}$
such that $a_{>i}\bar{b}_{<i}\ind\bar{b}_{i}$. 

We check that it satisfies (3): As all tuples are inside $p(\mathbb{M})$,
we can use symmetry, transitivity and $\ind^{d}=\ind^{f}$ freely.
And so, $a_{>i+1}a_{i+1}\bar{b}_{<i}\ind\bar{b}_{i}$ $\Rightarrow$
$a_{>i+1}\bar{b}_{<i}\ind_{a_{i+1}}\bar{b}_{i}$ $+$ $a_{>i+1}\bar{b}_{<i}\ind a_{i+1}$(as
$a_{>i+1}\ind a_{i+1}$ and $\bar{b}_{<i}\ind a_{\geq i+1}$ by (4)
for $i-1$) $\Rightarrow$ $a_{>i+1}\bar{b}_{<i}\ind\bar{b}_{i}a_{i+1}$
$\Rightarrow$ $a_{>i+1}\bar{b}_{<i}\ind_{\bar{b}_{i}}a_{i+1}$ $+$
$\bar{b}_{i}\ind a_{i+1}$ $\Rightarrow$ $a_{>i+1}\bar{b}_{\leq i}\ind a_{i+1}$.

We check that it satisfies (4): As $a_{>i}\bar{b}_{<i}\ind\bar{b}_{i}$
$\Rightarrow$ $a_{>i}\ind_{\bar{b}_{<i}}\bar{b}_{i}$ $+$ $a_{>i}\ind\bar{b}_{<i}$
by (4) for $i-1$ $\Rightarrow$ $a_{>i}\ind\bar{b}_{\leq i}$.

~

Having chosen $(\bar{b}_{i})_{i<\omega}$ we see that they are almost
mutually indiscernible by (1) and (2). Conclude by Lemma \ref{lem: almost indiscernible array gives an indiscernible array}.\end{proof}
\begin{lem}
\label{lem: In simple type in NTP2 theory, there is a MS witnessing dividing}
Let $T$ be $\mbox{NTP}_{2}$, $A$ an extension base and $p(x)\in S(A)$
simple. Assume that $\phi(x,a)$ divides over $A$, with $a\models p(x)$.
Then there is a Morley sequence over $A$ witnessing it.\end{lem}
\begin{proof}
As $A$ is an extension base, let $M\supseteq A$ be such that $M\ind_{A}^{f}a$.
Then $\phi(x,a)$ divides over $M$. By Fact \ref{Fac: Forking in an NTP2 theory}(1),
there is a Morley sequence $(a_{i})_{i<\omega}$ over $M$ witnessing
it (in particular $(a_{i})_{i<\omega}\subseteq p(\mathbb{M})$). We
show that it is actually a Morley sequence over $A$. Indiscernibility
is clear, and we check that $a_{i}\ind_{A}a_{<i}$ by induction. As
$a_{i}\ind_{M}a_{<i}$, $a_{<i}\ind_{M}a_{i}$ by simplicity of $\mbox{tp}(a_{<i}/M)$.
Noticing that $M\ind_{A}a_{i}$, we conclude $a_{<i}\ind_{A}a_{i}$,
so again by simplicity $a_{i}\ind_{A}a_{<i}$. \end{proof}
\begin{prop}
\label{thm: simple implies cosimple in NTP2} Let $T$ be $\mbox{NTP}_{2}$,
$A$ an extension base and $p(x)\in S(A)$ simple. Assume that $a\models p$
and $a\ind_{A}^{f}b$. Then $b\ind_{A}^{d}a$. \end{prop}
\begin{proof}
Assume that there is $\phi(x,a)\in L(Aa)$ such that $\models\phi(b,a)$
and $\phi(x,a)$ divides over $A$. Let $(a_{i})_{i<\omega}$ be a
Morley sequence over $A$ starting with $a$. Assume that $\{\phi(x,a_{i})\}_{i<\omega}$
is consistent. Let $\bar{a}_{0}$ be a Morley sequence witnessing
that $\phi(x,a_{0})$ $k$-divides over $A$ (exists by Lemma \ref{lem: In simple type in NTP2 theory, there is a MS witnessing dividing}),
and let $\bar{a}_{i}$ be its image under an $A$-automorphism sending
$a_{0}$ to $a_{i}$. By Lemma \ref{lem: Making almost mutually indiscernible},
we find $\bar{a}_{i}'\equiv_{a_{i}A}\bar{a}_{i}$, such that $(\bar{a}_{i}')_{i<\omega}$
are mutually indiscernible. But then we have that $\{\phi(x,a_{i,\eta(i)})\}_{i<\omega}$
is consistent for any $\eta\in\omega^{\omega}$, while $\{\phi(x,a_{i,j})\}_{j<\omega}$
is $k$-inconsistent for any $i<\omega$ --- contradiction to $\mbox{NTP}_{2}$.

Now let $(a_{i})_{i<\omega}$ be a Morley sequence over $A$ starting
with $a$ and indiscernible over $Ab$. Then clearly $b\models\left\{ \phi(x,a_{i})\right\} _{i<\omega}$
for any $\phi(x,a)\in\tp(b/aA)$, so by the previous paragraph $b\ind_{A}^{d}a$.\end{proof}
\begin{lem}
\label{lem: non co-simple fails symmetry over some M} Let $p(x)$
be a partial type over $A$. Assume that $p(x)$ is not co-simple
over $A$. Then there is some $M\supseteq A$, $a\models p(x)$ and
$b$ such that $a\ind_{M}^{u}b$ but $b\nind_{M}^{d}a$.\end{lem}
\begin{proof}
So assume that $p(x)$ is not co-simple over $A$, then there is an
$L(A)$-formula $\phi(x,y)$ and $\left(a_{\eta}\right)_{\eta\in\omega^{<\omega}}\subseteq p(\M)$
witnessing the tree property. Let $T^{\Sk}$ be a Skolemization of
$T$, then of course $\phi(x,y)$ and $a_{\eta}$ still witness the
tree property. As in the proof of (5)$\Rightarrow$(1) in Definition
\ref{def: co-simple}, working in the sense of $T^{\Sk}$, we can
find an $Ab$-indiscernible sequence $\left(a_{i}\right)_{i<\omega+1}$
in $p(x)$ such that $\phi(x,a_{i})$ divides over $Aa_{<i}$ and
$b\models\left\{ \phi(x,a_{i})\right\} _{i<\omega+1}$. Let $I=\left(a_{i}\right)_{i<\omega}$
and $\Sk(AI)=M\models T$. It follows that $a_{\omega}\ind_{M}^{u}b$
(by indiscernibility) and that $b\nind_{M}^{d}a_{\omega}$ (as $M\subseteq\dcl(Aa_{<\omega})$)
--- also in the sense of $T$, as dividing is witnessed by an $L$-formula
$\phi\left(x,y\right)$.\end{proof}
\begin{thm}
Let $T$ be $\NTP_{2}$, $A$ an arbitrary set and assume that $p(x)$
over $A$ is simple. Then $p(x)$ is co-simple over $A$.\end{thm}
\begin{proof}
If $p(x)$ over $A$ is not co-simple over $A$, then by Lemma \ref{lem: non co-simple fails symmetry over some M}
we find some $M\supseteq A$, $a\models p$ and $b$ such that $a\ind_{M}^{u}b$,
but $b\nind_{M}^{d}a$. As $M$ is an extension base, it follows by
Proposition \ref{thm: simple implies cosimple in NTP2} that $\tp(a/M)$
is not simple, thus $p(x)$ is not simple by Remark \ref{rem: tuples are simple}(1)
--- a contradiction.\end{proof}
\begin{cor}
\label{cor: simple types have full symmetry} Let $T$ be $\mbox{NTP}_{2}$
and $p(x)\in S(A)$ simple.
\begin{enumerate}
\item If $a\models p(x)$ then $a\ind_{A}b\Leftrightarrow b\ind_{A}a$
\item Right transitivity: If $a\models p(x)$, $B\supseteq A$, $a\ind_{A}B$
and $a\ind_{B}C$ then $a\ind_{A}C$.
\end{enumerate}
\end{cor}

\subsection{Independence and co-independence theorems.\protect \\
}

In \cite{MR1833481} Kim demonstrates that if $T$ has $\TP_{1}$,
then the independence theorem fails for types over models, assuming
the existence of a large cardinal. We give a proof of a localized
and a dual versions, showing in particular that the large cardinal
assumption is not needed.

\begin{defn}
Let $p(x)$ be (partial) type over $A$.
\begin{enumerate}
\item We say that $p(x)$ \emph{satisfies the independence theorem} if for
any $b_{1}\ind_{A}^{f}b_{2}$ and $c_{1}\equiv_{A}^{\mbox{\ensuremath{\lstp}}}c_{2}\subseteq p(\mathbb{M})$
such that $c_{1}\ind_{A}^{f}b_{1}$ and $c_{2}\ind_{A}^{f}b_{2}$,
there is some $c\ind_{A}^{f}b_{1}b_{2}$ such that $c\equiv_{b_{1}A}c_{1}$
and $c\equiv_{b_{2}A}c_{2}$.
\item We say that $p(x)$ \emph{satisfies the co-independence theorem} if
for any $b_{1}\ind_{A}^{f}b_{2}$ and $c_{1}\equiv_{A}^{\lstp}c_{2}\models p$
such that $b_{1}\ind_{A}^{f}c_{1}$ and $b_{2}\ind_{A}^{f}c_{2}$
, there is some $c\models p$ such that $b_{1}b_{2}\ind_{A}^{f}c$
and $c\equiv_{Ab_{1}}c_{1}$, $c\equiv_{Ab_{2}}c_{2}$.
\end{enumerate}
\end{defn}
Of course, both the independence and the co-independence theorems
hold in simple theories, but none of them characterizes simplicity.
\begin{prop}
\label{prop: co-independence theorem implies simple} Let $T$ be
$\mbox{NTP}_{2}$ and $p(x)$ is a partial type over $A$.
\begin{enumerate}
\item If every $p'(x)\supseteq p$ with $p'(x)\in S(M)$, $M\supseteq A$
satisfies the co-independence theorem, then it is simple.
\item If $p(x)$ satisfies the independence theorem, then it is co-simple.
\end{enumerate}
\end{prop}
\begin{proof}
(1) Our argument is based on the proof of \cite[Proposition 2.5]{MR1833481}.
Without loss of generality $A=\emptyset$. Assume that $p$ is not
simple, then by Fact \ref{fac: TP =00003D TP1 or TP2} there are some
formula $\phi(x,y)$ , $\left(a_{\eta}\right)_{\eta\in\omega^{<\omega}}$
such that:
\begin{itemize}
\item $\left\{ \phi(x,a_{\eta|i})\right\} _{i\in\omega}\cup p(x)$ is consistent
for every $\eta\in\omega^{\omega}$.
\item $\phi(x,a_{\eta})\land\phi(x,a_{\eta'})$ is inconsistent for any
incomparable $\eta,\eta'\in\omega^{<\omega}$.
\end{itemize}
By compactness we can find a tree with the same properties indexed
by $\kappa^{<\kappa}$, for a cardinal $\kappa$ large enough. Let
$T^{\Sk}$ be some Skolemization of $T$, and we work in the sense
of $T^{\Sk}$.
\begin{claim*}
There is a sequence $\left(c_{i}d_{i}\right)_{i\in\omega}$ satisfying:
\begin{enumerate}
\item $\left\{ \phi(x,c_{i})\right\} _{i\in\omega}\cup p(x)$ is consistent.
\item $c_{i},d_{i}$ start an infinite sequence indiscernible over $c_{<i}d_{<i}$.
\item $\phi(x,d_{i})\land\phi(x,d_{j})$ is inconsistent for any $i\neq j\in\omega$.
\end{enumerate}
\end{claim*}
\begin{proof}
By induction we choose $s_{i}\neq t_{i}\in\kappa$, $c_{i}=a_{s_{1}...s_{i-1}s_{i}}$
and $d_{i}=a_{s_{1}...s_{i-1}t_{i}}$ for some $s_{i}\neq t_{i}\in\kappa$
such that there is a $c_{<i}d_{<i}$-indiscernible sequence starting
with $a_{s_{1}...s_{i-1}s_{i}},a_{s_{1}...s_{i-1}t_{i}}$ (exists
by Erd\H os-Rado as $\kappa$ is large enough), so we get (2). From
the assumption on $(a_{\eta})_{\eta\in\kappa^{<\kappa}}$ we get (1)
as $s_{0}\vartriangleleft s_{0}s_{1}\vartriangleleft s_{0}s_{1}s_{2}\vartriangleleft\ldots$
lie on the same branch in the tree order and (3) as $s_{0}\ldots s_{i-1}t_{i}$
and $s_{0}\ldots s_{i-1}s_{i}$ are incomparable in the tree order.
\end{proof}
By compactness and Ramsey we can find $a$ and $\left(c_{i}d_{i}\right)_{i\leq\omega+1}$
indiscernible over $a$, satisfying (1)--(3) and such that $a\models p(x)\cup\left\{ \phi(x,c_{i})\right\} $.

Let $M=\Sk(c_{i}d_{i})_{i<\omega}$, a model of $T^{\Sk}$. Then we
have $c_{\omega+1}\ind_{M}^{u}a$ and $d_{\omega}\ind_{M}^{u}c_{\omega+1}$
by indiscernibility. As $c_{\omega}d_{\omega}$ start an $M$-indiscernible
sequence, there is $\sigma\in Aut(\M/M)$ sending $c_{\omega}$ to
$d_{\omega}$. Let $a'=\sigma(a)$, then $a'\equiv_{M}^{\lstp}a$,
$d_{\omega}\ind_{M}^{u}a'$ (as $c_{\omega}\ind_{M}^{u}a$ by indiscernibility)
and $\phi(a',d_{\omega})$. But $\phi(x,c_{\omega+1})\land\phi(x,d_{\omega})$
is inconsistent by (3)+(2). As $\phi$ is an $L$-formula, $M$ is
in particular an $L$-model and $\ind^{u}$ in the sense of $T^{\Sk}$
implies $\ind^{u}$ in the sense of $T$, we get that the co-independence
theorem fails for $p'=\tp_{L}(a/M)$ in $T$.

(2) Similar.
\end{proof}
Now we will show that in $\NTP_{2}$ theories simple types satisfy
the independence theorem over extension bases. We will need the following
fact from \cite{CheBY}.
\begin{fact}
\label{fac: Udi's weak independence theorem} Let $T$ be $\NTP_{2}$
and $M\models T$. Assume that $c\ind_{M}ab$, $b\ind_{M}a$, $b'\ind_{M}a$,
$b\equiv_{M}b'$. Then there exists $c'\ind_{M}ab'$ and $c'b'\equiv_{M}cb$,
$c'a\equiv_{M}ca$.
\end{fact}

\begin{prop}
\label{prop: simple types satisfy independence theorem} Let $T$
be $\NTP_{2}$ and $p(x)$ a simple type over $M\models T$. Then
it satisfies the independence theorem: assume that $e_{1}\ind_{M}e_{2}$,
$d_{i}\ind_{M}e_{i}$, $d_{1}\equiv_{M}d_{2}\models p(x)$. Then there
is $d\ind_{M}e_{1}e_{2}$ with $d\equiv_{e_{i}M}d_{i}$.\end{prop}
\begin{proof}
First we find some $e_{1}'\ind_{M}d_{2}e_{2}$ and such that $e_{1}'d_{2}\equiv_{M}e_{1}d_{1}$
(Let $\sigma\in\Aut(\M/M)$ be such that $\sigma(d_{1})=d_{2}$, then
$\sigma(e_{1})d_{2}\equiv_{M}e_{1}d_{1}$. By simplicity of $\tp(d_{1}/M)$
and the assumption we get $e_{1}\ind_{M}d_{1}$, which implies that
$\sigma(e_{1})\ind_{M}d_{2}$. Let $e_{1}'$ realize a non-forking
extension to $d_{2}e_{2}$). Then we also have $d_{2}\ind_{M}e_{1}'e_{2}$
(by transitivity and symmetry using simplicity of $\tp(d_{2}/M)$).

Applying Fact \ref{fac: Udi's weak independence theorem} with $a=e_{2},b=e_{1}',b'=e_{1},c=d_{2}$
we find some $d\ind_{M}e_{1}e_{2}$, $de_{1}\equiv_{M}d_{2}e_{1}'\equiv_{M}d_{1}e_{1}$
and $de_{2}\equiv_{M}d_{2}e_{2}$ --- as wanted.
\end{proof}
We conclude with the main theorem of the section.
\begin{thm}
Let $T$ be $\NTP_{2}$ and $p(x)$ a partial type over $A$. Then
the following are equivalent:
\begin{enumerate}
\item $p(x)$ is simple (in the sense of Definition \ref{def: simple}).
\item For any $B\supseteq A$, $a\models p$ and $b$, $a\ind_{A}^{f}b$
if and only if $b\ind_{A}^{f}a$.
\item Every extension $p'(x)\supseteq p(x)$ to a model $M\supseteq A$
satisfies the co-independence theorem.
\end{enumerate}
\end{thm}
\begin{proof}
(1) is equivalent to (2) is by Definitions \ref{def: simple} and
Corollary \ref{cor: simple types have full symmetry}.

(1) implies (3): By Proposition \ref{prop: simple types satisfy independence theorem}
and Corollary \ref{cor: simple types have full symmetry}.

(3) implies (1) is by Proposition \ref{prop: co-independence theorem implies simple}.\end{proof}
\begin{problem}
Is every co-simple type simple in an $\NTP_{2}$ theory?
\end{problem}
We point out that at least every co-simple \emph{stably embedded}
type (defined over a small set) is simple. Recall that a partial type
$p(x)$ defined over $A$ is called\emph{ }stably embedded if for
any $\phi(\bar{x},c)$ there is some $\psi(\bar{x},y)\in L(A)$ and
$d\in p(\M)$ such that $p(\M)^{n}\cap\phi(\bar{x},c)=p(\M)^{n}\cap\psi(\bar{x},d)$.
If $p(x)$ happens to be defined by finitely many formulas, it is
easy to see by compactness that $\psi(\bar{x},y)$ can be chosen to
depend just on $\phi(\bar{x},y)$, and not on $c$. But for an arbitrary
type this is not true.
\begin{prop}
Let $T$ be $\NTP_{2}$. Let $p(x)$ be a co-simple type over $A$
and assume that $p$ is stably embedded. Then $p(x)$ is simple.\end{prop}
\begin{proof}
Assume $p(x)$ is not simple, and let $(a_{\eta})_{\eta\in\omega^{<\omega}}$,
$k$ and $\phi(x,y)$ witness this. We may assume in addition that
$(a_{\eta})$ is an indiscernible tree over $A$ (that is, ss-indiscernible
in the terminology of \cite{TreeIndiscernibility}, see Definition
3.7 and the proof of Theorem 6.6 there).

By the stable embeddedness assumption, there is some $\psi(x,z)\in L(A)$
and $b\subseteq p(\M)$ such that $\psi(x,b)\cap p(\M)=\phi(x,a_{\emptyset})\cap p(\M)$.
It follows by the indiscernibility over $A$ that for every $\eta\in\omega^{<\omega}$
there is $b_{\eta}\subseteq p(\M)$ satisfying $\psi(x,b_{\eta})\cap p(\M)=\phi(x,a_{\eta})\cap p(\M)$. 

As $\left\{ \phi(x,a_{\emptyset i})\right\} _{i<\omega}$ is $k$-inconsistent,
it follows that $\left\{ \psi(x,b_{\emptyset i})\right\} _{i<\omega}\cup p(x)$
is $k$-inconsistent, thus $\left\{ \psi(x,b_{\emptyset i})\right\} _{i<\omega}\cup\left\{ \chi(x)\right\} $
is $k$-inconsistent for some $\chi(x)\in p$ by compactness and indiscernibility.
Again by the indiscernibility over $A$ we have that $\left\{ \psi(x,b_{\eta i})\right\} _{i<\omega}\cup\left\{ \chi(x)\right\} $
is $k$-inconsistent for every $\eta\in\omega^{<\omega}$. It is now
easy to see that $\psi'(x,z)=\psi(x,z)\land\chi(x)$ and $\left(b_{\eta}\right)_{\eta\in\omega^{<\omega}}$
witness that $p(x)$ is not co-simple over $A$.\end{proof}
\begin{rem}
If $p(x)$ is actually a definable set, the argument works in an arbitrary
theory since instead of extracting a sufficiently indiscernible tree
(which seems to require $\NTP_{2}$), we just use the uniformity of
stable embeddedness given by compactness.
\end{rem}

\section{\label{sec: Examples}Examples}

In this section we present some examples of $\NTP_{2}$ theories.
But first we state a general lemma which may sometimes simplify checking
$\NTP_{2}$ in particular examples.
\begin{lem}
\label{lem: boolean operations on inp-patterns} ~
\begin{enumerate}
\item If $(\bar{a}_{\alpha},\phi_{\alpha,0}(x,y_{\alpha,0})\lor\phi_{\alpha,1}(x,y_{\alpha,1}),\, k_{\alpha})_{\alpha<\kappa}$
is an $\inp$-pattern, then $(\bar{a}_{\alpha},\phi_{\alpha,f(\alpha)}(x,y_{\alpha,f(\alpha)}),$
$k_{\alpha})_{\alpha<\kappa}$ is an $\inp$-pattern for some $f:\,\kappa\to\{0,1\}$.
\item Let $\left(\bar{a}_{\alpha},\phi_{\alpha}(x,y_{\alpha}),k_{\alpha}\right)_{\alpha<\kappa}$
be an $\inp$-pattern and assume that $\phi_{\alpha}(x,a_{\alpha0})\leftrightarrow\psi_{\alpha}(x,b_{\alpha})$
for $\alpha<\kappa$. Then there is an $\inp$-pattern of the form
$\left(\bar{b}_{\alpha},\psi_{\alpha}(x,z_{\alpha}),k_{\alpha}\right)_{\alpha<\kappa}$.
\end{enumerate}
\end{lem}

\subsection{Adding a generic predicate}

Let $T$ be a first-order theory in the language $L$. For $S(x)\in L$
we let $L_{P}=L\cup\left\{ P(x)\right\} $ and $T_{P,S}^{0}=T\cup\left\{ \forall x\left(P(x)\rightarrow S(x)\right)\right\} $. 
\begin{fact}
\label{fac: properties of T_P}\cite{ChatzidakisPillay} Let $T$
be a theory eliminating quantifiers and $\exists^{\infty}$. Then:
\begin{enumerate}
\item $T_{P,S}^{0}$ has a model companion $T_{P,S}$, which is axiomatized
by $T$ together with 
\begin{eqnarray*}
 & \forall\bar{z}\left[\exists\bar{x}\phi(\bar{x},\bar{z})\land\left(\bar{x}\cap\acl_{L}(\bar{z})=\emptyset\right)\land\bigwedge_{i<n}S(x_{i})\land\bigwedge_{i\neq j<n}x_{i}\neq x_{j}\right]\rightarrow\\
 & \left[\exists\bar{x}\phi(\bar{x},\bar{z})\land\bigwedge_{i\in I}P(x_{i})\land\bigwedge_{i\notin I}\neg P(x_{i})\right]
\end{eqnarray*}
for every formula $\phi(\bar{x},\bar{z})\in L$, $\bar{x}=x_{0}...x_{n-1}$
and every $I\subseteq n$. It is possible to write it in first-order
due to the elimination of $\exists^{\infty}$.
\item $\acl_{L}(a)=\acl_{L_{P}}(a)$
\item $a\equiv^{L_{P}}b$ $\Leftrightarrow$ there is an isomorphism between
$L_{P}$ structures $f:\acl(a)\to\acl(b)$ such that $f(a)=b$.
\item Modulo $T_{P,S}$, every formula $\psi(\bar{x})$ is equivalent to
a disjunction of formulas of the form $\exists\bar{z}\phi(\bar{x},\bar{z})$
where $\phi(\bar{x},\bar{z})$ is a quantifier-free $L_{P}$ formula
and for any $\bar{a},\bar{b}$, if $\models\phi(\bar{a},\bar{b})$,
then $\bar{b}\in\acl(\bar{a})$.
\end{enumerate}
\end{fact}
\begin{thm}
Let $T$ be geometric (that is, the algebraic closure satisfies the
exchange property, and $T$ eliminates $\exists^{\infty}$) and $\NTP_{2}$.
Then $T_{P}$ is $\NTP_{2}$.\end{thm}
\begin{proof}
Denote $a\ind_{c}^{a}b$ $\Leftrightarrow$ $a\notin\acl(bc)\setminus\acl(c)$.
As $T$ is geometric, $\ind^{a}$ is a symmetric notion of independence,
which we will be using freely from now on.

Let $\left(\bar{a}_{i},\phi(x,y),k\right)_{i<\omega}$ be an $\inp$-pattern,
such that $\left(\bar{a}_{i}\right)_{i<\omega}$ is an indiscernible
sequence and $\bar{a}_{i}$'s are mutually indiscernible in the sense
of $L_{P}$, and $\phi$ an $L_{P}$-formula.
\begin{claim*}
For any $i$, $\left\{ a_{ij}\right\} _{j<\omega}$ is an $\ind^{a}$-independent
set (over $\emptyset$) and $a_{ij}\notin\acl(\emptyset)$.\end{claim*}
\begin{proof}
By indiscernibility and compactness.
\end{proof}
Let $A=\bigcup_{i<\omega}\bar{a}_{i}$.
\begin{claim*}
There is an infinite $A$-indiscernible sequence $\left(b_{t}\right)_{t<\omega}$
such that $b_{t}\models\left\{ \phi(x,a_{i0})\right\} _{i<\omega}$
for all $t<\omega$. \end{claim*}
\begin{proof}
First, there are infinitely many different $b_{t}$'s realizing $\left\{ \phi(x,a_{i0})\right\} _{i<\omega}$,
as $\left\{ \phi(x,a_{i0})\right\} _{0<i<\omega}\cup\left\{ \phi(x,a_{0j})\right\} $
is consistent for any $j<\omega$ and $\left\{ \phi(x,a_{0j})\right\} _{j<\omega}$
is $k$-inconsistent. Extract an $A$-indiscernible sequence from
it.
\end{proof}
Let $p_{i}(x,a_{i0})=\tp_{L}(b_{0}/a_{i0})$. 
\begin{claim*}
For some/every $i<\omega$, there is $b\models\bigcup_{j<\omega}p_{i}(x,a_{ij})$
such that in addition $b\notin\acl(A)$.\end{claim*}
\begin{proof}
For any $N<\omega$, let 
\[
q_{i}^{N}(x_{0}...x_{N-1},a_{i0})=\bigcup_{n<N}p_{i}(x_{n},a_{i0})\cup\left\{ x_{n_{1}}\neq x_{n_{2}}\right\} _{n_{1}\neq n_{2}<N}
\]
As $b_{0}...b_{N-1}\models\bigcup_{i<\omega}q_{i}^{N}(x_{0}...x_{N-1},a_{i0})$
and $T$ is $\NTP_{2}$, there must be some $i<\omega$ such that
$\bigcup_{j<\omega}q_{i}^{N}(x_{0}...x_{N-1},a_{ij})$ is consistent
for arbitrary large $N$ (and by indiscernibility this holds for every
$i$). Then by compactness we can find $b\models\bigcup_{j<\omega}p_{i}(x,a_{ij})$
such that in addition $b\notin\acl(A)$.
\end{proof}
Work with this fixed $i$. Notice that $b_{0}a_{i0}\equiv^{L}ba_{ij}$
for all $j\in\omega$.
\begin{claim*}
The following is easy to check using that $\ind^{a}$ satisfies exchange.
\begin{enumerate}
\item $\acl(A)\cap\acl(a_{ij}b)=\acl(a_{ij})$.
\item $\acl(a_{ij}b)\cap\acl(a_{ik}b)=\acl(b)$ for $j\neq k$.
\end{enumerate}
\end{claim*}
Now we conclude as in the proof of \cite[Theorem 2.7]{ChatzidakisPillay}.
That is, we are given a coloring $P$ on $\bar{a}_{i}$. Extend it
to a $P_{i}$-coloring on $\acl(a_{ij}b)$ such that $a_{ij}b$ realizes
$\tp_{L_{P}}(a_{i0}b_{0})$, and by the claim all $P_{i}$'s are consistent.
Thus there is some $b'$ such that $b_{0}a_{i0}\equiv^{L_{P}}b'a_{ij}$
for all $j\in\omega$, in particular $b'\models\left\{ \phi_{i}(x,a_{ij})\right\} $
--- a contradiction.\end{proof}
\begin{example}
Adding a (directed) random graph to an $o$-minimal theory is $\NTP_{2}$.\end{example}
\begin{problem}
Is it true without assuming exchange for the algebraic closure? Is
$\kappa_{\inp}$ preserved? So in particular, is strongness preserved?
\end{problem}

\subsection{Valued fields}

In this section we are going to prove the following theorem:
\begin{thm}
\label{thm: Ax-Kochen} Let $\bar{K}=\left(K,\Gamma,k,v:\, K\to\Gamma,ac:\, K\to k\right)$
be a Henselian valued field of characteristic $(0,0)$ in the Denef-Pas
language. Let $\kappa=\kappa_{\inp}^{1}(k)\times\kappa_{\inp}^{1}(\Gamma)$.
Then $\kappa_{\inp}^{1}(K)<R(\kappa+2,\Delta)$ for some finite set
of formulas $\Delta$ (see Definition \ref{def: Ramsey number}).
In particular:
\begin{enumerate}
\item If $k$ is $\NTP_{2}$, then $\bar{K}$ is $\NTP_{2}$ (If $K$ was
$\TP_{2}$, then by Lemma \ref{lem: NTP_2 =00003D bounded burden}
we would have $\kappa_{\inp}^{1}\left(K\right)=\infty>\beth_{\omega}\left(\left|T\right|^{+}\right)>R\left(\left|T\right|^{+}+2,\Delta\right)$.
Every ordered abelian group is $\NIP$ by \cite{GurevichSchmitt},
thus $\kappa_{\inp}(\Gamma)\leq\left|T\right|$. But then the theorem
implies $\kappa_{\inp}^{1}\left(k\right)>\left|T\right|^{+}$, so
$k$ has $\TP_{2}$).
\item If $k$ and $\Gamma$ are strong (of finite burden), then $\bar{K}$
is strong (resp. of finite burden). The argument is the same as for
(1) using Definition \ref{def: Ramsey number}(1),(2).
\end{enumerate}
\end{thm}
\begin{example}

\begin{enumerate}
\item Hahn series over pseudo-finite fields are $\NTP_{2}$.
\item In particular, let $K=\prod_{p\mbox{ prime}}\mathbb{Q}_{p}/\mathfrak{U}$
with $\mathfrak{U}$ a non-principal ultra-filter. Then $k$ is pseudo-finite,
so has $\IP$ by \cite{DuretPAChasIP}. And $\Gamma$ has SOP of course.
It is known that the valuation rings of $\mathbb{Q}_{p}$ are definable
in the pure field language uniformly in $p$ (see e.g. \cite{AxDefiningValuation}),
thus the valuation ring is definable in $K$ in the pure field language,
so $K$ has both IP and SOP in the pure field language. By Theorem
\ref{thm: Ax-Kochen} it is strong of finite burden, even in the larger
Denef-Pas language. Notice, however, that the burden of $K$ is at
least $2$ (witnessed by the formulas ``$ac(x)=y$'', ``$v(x)=y$''
and infinite sequences of different elements in $k$ and $\Gamma$.
\end{enumerate}
\end{example}
\begin{cor}
\cite{Sh863} If $k$ and $\Gamma$ are strongly dependent, then $K$
is strongly dependent.\end{cor}
\begin{proof}
By Delon's theorem \cite{Delon}, if $k$ is $\NIP$, then $K$ is
$\NIP$. Conclude by Theorem \ref{thm: Ax-Kochen} and Fact \ref{fac: In NIP, burden =00003D dp-rank}.
\end{proof}
We start the proof with a couple of lemmas about the behavior of $v(x)$
and $ac(x)$ on indiscernible sequences which are easy to check.
\begin{lem}
\label{lem: val on indiscernible sequence} Let $(c_{i})_{i\in I}$
be indiscernible. Consider function $(i,j)\mapsto v(c_{j}-c_{i})$
with $i<j$. It satisfies one of the following:
\begin{enumerate}
\item It is strictly increasing depending only on $i$ (so the sequence
is pseudo-convergent).
\item It is strictly decreasing depending only on $j$ (so the sequence
taken in the reverse direction is pseudo-convergent).
\item It is constant (we'll call such a sequence ``constant'').
\end{enumerate}
\end{lem}
Contrary to the usual terminology we do not exclude index sets with
a maximal element.

\begin{lem}
\label{lem: Phase change point} Let $(c_{i})_{i\in I}$ be an indiscernible
pseudo-convergent sequence. Then for any $a$ there is some $h\in\bar{I}\cup\{+\infty,-\infty\}$
(where $\bar{I}$ is the Dedekind closure of $I$) such that (taking
$c_{\infty}$ such that $I\frown c_{\infty}$ is indiscernible):
\begin{lyxlist}{00.00.0000}
\item [{For~$i<h$:}] $v(c_{\infty}-c_{i})<v(a-c_{\infty})$,$v(a-c_{i})=v(c_{\infty}-c_{i})$
and $ac(a-c_{i})=ac(c_{\infty}-c_{i})$.
\item [{For~$i>h$:}] $v(c_{\infty}-c_{i})>v(a-c_{\infty})$, $v(a-c_{i})=v(a-c_{\infty})$
and $ac(a-c_{i})=ac(a-c_{\infty})$.
\end{lyxlist}
\end{lem}
Notice that in fact there is a finite set of formulas $\Delta$ such
that these lemmas are true for $\Delta$-indiscernible sequences.
Fix it from now on, and let $\delta=R(\kappa+2,\Delta)$ for $\kappa=\kappa_{k}\times\kappa_{\Gamma}$
with $\kappa_{k}=\kappa_{\inp}^{1}(k)$ and $\kappa_{\Gamma}=\kappa_{\inp}^{1}(\Gamma)$.
\begin{lem}
\label{lem: 1-lin formulas are NTP2} In $K$, there is no $\inp$-pattern
$\left(\phi_{\alpha}(x,y_{\alpha}),\bar{d}_{\alpha},k_{\alpha}\right)_{\alpha<\delta}$
with mutually indiscernible rows such that $x$ is a singleton and
$\phi_{\alpha}(x,y_{\alpha})=\chi_{\alpha}(v(x-y),y_{\alpha}^{\Gamma})\land\rho_{\alpha}(ac(x-y),y_{\alpha}^{k})$,
where $\chi_{\alpha}\in L_{\Gamma}$ and $\rho_{\alpha}\in L_{k}$.\end{lem}
\begin{proof}
Assume otherwise, and let $d_{\alpha i}=c_{\alpha i}d_{\alpha i}^{\Gamma}d_{\alpha i}^{k}$
where $c_{\alpha i}\in K$ corresponds to $y$, $d_{\alpha i}^{\Gamma}\in\Gamma$
corresponds to $y_{\alpha}^{\Gamma}$ and $d_{\alpha i}^{k}\in k$
corresponds to $y_{\alpha}^{k}$. By the choice of $\delta$, there
is a $\Delta$-indiscernible sub-sequence of $\left(c_{\alpha0}\right)_{\alpha<\delta}$
of length $\kappa+2$. Take a sub-array consisting of rows starting
with these elements -- it is still an $\inp$-pattern of depth $\kappa+2$
-- and replace our original array with it. Let $c_{-\infty}$ and
$c_{\infty}$ be such that $c_{-\infty}\frown\left(c_{\alpha0}\right)_{\alpha<\kappa}\frown c_{\infty}$
is $\Delta$-indiscernible and $\left(\bar{d}_{\alpha}\right)_{\alpha<\kappa}$
is a mutually indiscernible array over $c_{-\infty}c_{\infty}$ (so
either find $c_{\infty}$ by compactness if $\kappa$ is infinite,
or just let it be $c_{\kappa-1,0}$ and replace our array by $\left(\bar{d}_{\alpha}\right)_{\alpha<\kappa-1}$).
Let $a\models\left\{ \phi_{\alpha}(x,d_{\alpha0})\right\} _{\alpha<\kappa+1}$.

\textbf{Case 1}. $\left(c_{\alpha0}\right)$ is pseudo-convergent.
Let $h\in\{-\infty\}\cup\kappa+1\cup\{\infty\}$ be as given by Lemma
\ref{lem: Phase change point}.

\emph{Case 1.1}. Assume $0<h$. Then $v(a-c_{00})=v(c_{\infty}-c_{00})$,
$ac(a-c_{00})=ac(c_{\infty}-c_{00})$. But then actually $c_{\infty}\models\phi(x,d_{00})$,
and by indiscernibility of the array over $c_{\infty}$, $c_{\infty}\models\{\phi(x,d_{0i})\}_{i<\omega}$
--- a contradiction.

\textit{Case 1.2}: Thus $v(a-c_{\alpha0})=v(a-c_{\infty})$, $ac(a-c_{\alpha0})=ac(a-c_{\infty})$
and $v(a-c_{\infty})<v(c_{\infty}-c_{\alpha0})$ for all $0<\alpha<\kappa+1$.

Let $\chi_{\alpha}'(x',e_{\alpha i}^{\Gamma}):=\chi_{\alpha}(x',d_{\alpha i}^{\Gamma})\land x'<v(c_{\infty}-c_{\alpha i})$
with $e_{\alpha i}^{\Gamma}=d_{\alpha i}^{\Gamma}\cup v(c_{\infty}-c_{\alpha i})$.
Finally, for $\alpha<\kappa_{\Gamma}$ let $f_{\alpha i}^{\Gamma}=\bigcup_{\beta<\kappa_{k}}e_{\kappa_{k}\times\alpha+\beta,i}$
and $p_{\alpha}(x',f_{\alpha i}^{\Gamma})=\left\{ \chi_{\beta}'(x',e_{\kappa_{k}\times\alpha+\beta,i}^{\Gamma})\right\} _{\beta<\kappa_{k}}$.
As $\left(f_{\alpha i}^{\Gamma}\right)$ is a mutually indiscernible
array in $\Gamma$, $\left\{ p_{\alpha}(x',f_{\alpha0}^{\Gamma})\right\} _{\alpha<\kappa_{\Gamma}}$
is realized by $v(a-c_{\infty})$ and $\kappa_{\inp}^{1}(\Gamma)=\kappa_{\Gamma}$,
there must be some $\alpha<\kappa_{\Gamma}$ and $a_{\Gamma}\in\Gamma$
such that (unwinding) $a_{\Gamma}\models\left\{ \chi_{\beta}'(x',e_{\kappa_{k}\times\alpha+\beta,i}^{\Gamma})\right\} _{\beta<\kappa_{k},i<\omega}$. 

Analogously letting $\chi'_{\beta}(x',e_{\beta i}^{k}):=\rho_{\kappa_{k}\times\alpha+\beta}(x',d_{\kappa_{k}\times\alpha+\beta,i}^{k})$,
noticing that $(e_{\beta i}^{k})_{\beta<\kappa_{k},i<\omega}$ is
an indiscernible array in $k$ and $\kappa_{k}=\kappa_{\inp}(k)$,
there must be some $a_{\rho}\in k$ and $\beta<\kappa_{k}$ such that
$a_{\rho}\models\{\chi'_{\beta}(x',e_{\beta i}^{k})\}_{i<\omega}$.

Finally, take $a'\in K$ with $v(a'-c_{\infty})=a_{\Gamma}\land ac(a'-c_{\infty})=a_{\rho}$
and let $\gamma=\kappa_{k}\times\alpha+\beta$. As $a_{\Gamma}<v(c_{\infty}-c_{\gamma i})$
it follows that $v(a'-c_{\gamma i})=v(a'-c_{\infty})$ and $ac(a'-c_{\gamma i})=ac(a'-c_{\infty})$.
But then $a'\models\left\{ \phi_{\gamma}(x,d_{\gamma i})\right\} _{i<\omega}$
--- a contradiction.

\textbf{Case 2}: $(c_{0}^{\alpha})$ is decreasing --- reduces to
the first case by reversing the order of rows.

\textbf{Case 3}: $(c_{0}^{\alpha})$ is constant. 

If $v(a-c_{\alpha0})<v(c_{\infty}-c_{\alpha0})$ ($=v(c_{\beta0}-c_{\alpha0})$
for $\beta\neq\alpha$) for some $\alpha$, then $v(a-c_{\alpha0})=v(a-c_{\beta0})=v(a-c_{\infty})$
for any $\beta$, and $ac(a-c_{\alpha0})=ac(a-c_{\infty})$ for all
$\alpha$'s and it falls under case 1.2.

Next, there can be at most one $\alpha$ with $v(a-c_{\alpha0})>v(c_{\infty}-c_{\alpha0})$
(if also $v(a-c_{\beta0})>v(c_{\infty}-c_{\beta0})$ for some $\beta>\alpha$
then $v(c_{\infty}-c_{\beta0})=v(c_{\beta0}-c_{\alpha0})=v(a-c_{\beta0})>v(c_{\infty}-c_{\beta0})$,
a contradiction). Throw the corresponding row away and we are left
with the case $v(a-c_{\alpha0})=v(c_{\infty}-c_{\alpha0})=v(a-c_{\infty})$
for all $\alpha<\kappa$. It follows by indiscernibility that $v(a-c_{\infty})=v(c_{\infty}-c_{\alpha i})$
for all $\alpha,i$. Notice that it follows that $ac(a-c_{\alpha0})\neq ac(c_{\infty}-c_{\alpha0})$
and $ac(a-c_{\alpha0})=ac(a-c_{\infty})+ac(c_{\infty}-c_{\alpha0})$.

Let $\rho_{\alpha}'(x',e_{\alpha i}^{k}):=\rho_{\alpha}(x'-ac(c_{\infty}-c_{\alpha i}),d_{\alpha i}^{k})\land x'\neq ac(c_{\infty}-c_{\alpha i})$
with $e_{\alpha i}^{k}=d_{\alpha i}^{k}\cup ac(c_{\infty}-c_{\alpha i})$.
Notice that $ac(a-c_{\infty})\models\left\{ \rho_{\alpha}'(x',e_{\alpha0}^{k})\right\} $
and that $\left(e_{\alpha i}^{k}\right)$ is a mutually indiscernible
array in $k$. Thus there is some $\alpha<\kappa$ and $a_{k}\models\left\{ \rho_{\alpha}'(x',e_{\alpha i}^{k})\right\} _{i<\omega}$.

Take $a'\in K$ such that $v(a'-c_{\infty})=v(a-c_{\infty})\land ac(a'-c_{\infty})=a_{k}$.
By the choice of $a_{k}$ we have that $v(a'-c_{\infty})=v(a-c_{\infty})=v(c_{\infty}-c_{\alpha i})$
and that $ac(a'-c_{\infty})\neq ac(c_{\infty}-c_{\alpha i})$, thus
$v(a'-c_{\alpha i})=v(a-c_{\infty})$ and $ac(a'-c_{\alpha i})=a_{k}+ac(c_{\infty}-c_{\alpha i})$.
It follows that $a'\models\left\{ \phi_{\alpha}(x,d_{\alpha i})\right\} _{i<\omega}$
--- a contradiction.\end{proof}
\begin{lem}
\label{lem: linear formulas are NTP2} In $K$, there is no $\inp$-pattern
$\left(\phi_{\alpha}(x,y_{\alpha}),\bar{d}_{\alpha},k_{\alpha}\right)_{\alpha<\delta}$
such that $x$ is a singleton and $\phi_{\alpha}(x,y_{\alpha})=\chi_{\alpha}(v(x-y_{1}),...,v(x-y_{n}),y_{\alpha}^{\Gamma})\land\rho_{\alpha}(ac(x-y_{1}),...,ac(x-y_{n}),y_{\alpha}^{k})$,
where $\chi_{\alpha}\in L_{\Gamma}$ and $\rho_{\alpha}\in L_{k}$.\end{lem}
\begin{proof}
We prove it by induction on $n$. The base case is given by Lemma
\ref{lem: 1-lin formulas are NTP2}. So assume that we have proved
it for $n-1$, and let $\left(\phi_{\alpha}(x,y_{\alpha}),\bar{d}_{\alpha},k_{\alpha}\right)_{\alpha<\delta}$
be an $\inp$-pattern with $\phi_{\alpha}(x,y_{\alpha})=\chi_{\alpha}(v(x-y_{1}),...,v(x-y_{n}),y_{\alpha}^{\Gamma})\land\rho_{\alpha}(ac(x-y_{1}),...,ac(x-y_{n}),y_{\alpha}^{k})$
and $d_{\alpha i}=c_{\alpha i}^{1}...c_{\alpha i}^{n}d_{\alpha i}^{\Gamma}d_{\alpha i}^{k}$. 

So let $a\models\left\{ \phi_{\alpha}(x,d_{\alpha0})\right\} _{\alpha<\delta}$.
Fix some $\alpha<\delta$.

\textbf{Case 1}: $v(a-c_{\alpha0}^{1})<v(c_{\alpha0}^{n}-c_{\alpha0}^{1})$. 

Then $v(a-c_{\alpha0}^{1})=v(a-c_{\alpha0}^{n})$ and $ac(a-c_{\alpha0}^{1})=ac(a-c_{\alpha0}^{n})$.
We take 
\begin{eqnarray*}
\phi_{\alpha}'(x,d_{\alpha i}') & = & \left(\chi_{\alpha}(v(x-c_{\alpha i}^{1}),...,v(x-c_{\alpha i}^{1}),d_{\alpha i}^{\Gamma})\land v(x-c_{\alpha0}^{1})<v(c_{\alpha i}^{n}-c_{\alpha i}^{1})\right)\\
 &  & \land\rho_{\alpha}(ac(x-c_{\alpha i}^{1}),...,ac(x-c_{\alpha i}^{1}),d_{\alpha i}^{\rho})
\end{eqnarray*}
 and $d'_{\alpha i}=d_{\alpha i}\cup v(c_{\alpha i}^{n}-c_{\alpha i}^{1})$.

\textbf{Case 2}: $v(a-c_{\alpha0}^{1})>v(c_{\alpha0}^{n}-c_{\alpha0}^{1})$.

Then $v(a-c_{\alpha0}^{n})=v(c_{\alpha0}^{n}-c_{\alpha0}^{1})$ and
$ac(a-c_{\alpha0}^{n})=ac(c_{\alpha0}^{n}-c_{\alpha0}^{1})$. Take
\begin{eqnarray*}
\phi_{\alpha}'(x,d_{\alpha i}') & = & \left(\chi_{\alpha}(v(x-c_{\alpha i}^{1}),...,v(c_{\alpha0}^{n}-c_{\alpha0}^{1}),d_{\alpha i}^{\Gamma})\land v(x-c_{\alpha0}^{1})>v(c_{\alpha i}^{n}-c_{\alpha i}^{1})\right)\\
 &  & \land\rho_{\alpha}(ac(x-c_{\alpha i}^{1}),...,ac(c_{\alpha0}^{n}-c_{\alpha0}^{1}),d_{\alpha i}^{\rho})
\end{eqnarray*}
 and $d'_{\alpha i}=d_{\alpha i}\cup v(c_{\alpha i}^{n}-c_{\alpha i}^{1})\cup ac(c_{\alpha0}^{n}-c_{\alpha0}^{1})$.

\textbf{Case 3}: $v(a-c_{\alpha0}^{n})<v(c_{\alpha0}^{n}-c_{\alpha0}^{1})$
and \textbf{Case 4}: $v(a-c_{\alpha0}^{n})>v(c_{\alpha0}^{n}-c_{\alpha0}^{1})$
are symmetric to the cases 1 and 2, respectively.

\textbf{Case 5}: $v(a-c_{\alpha0}^{1})=v(a-c_{\alpha0}^{n})=v(c_{\alpha0}^{n}-c_{\alpha0}^{1})$.

Then $ac(a-c_{\alpha0}^{n})=ac(a-c_{\alpha0}^{1})-ac(c_{\alpha0}^{n}-c_{\alpha0}^{1})$.
We take 

\begin{eqnarray*}
\phi_{\alpha}'(x,d_{\alpha i}') & = & \left(\chi_{\alpha}(v(x-c_{\alpha i}^{1}),...,v(c_{\alpha0}^{n}-c_{\alpha0}^{1}),d_{\alpha i}^{\Gamma})\land v(x-c_{\alpha0}^{1})=v(c_{\alpha i}^{n}-c_{\alpha i}^{1})\right)\\
 &  & \land\left(\rho_{\alpha}(ac(x-c_{\alpha i}^{1}),...,ac(c_{\alpha0}^{n}-c_{\alpha0}^{1}),d_{\alpha i}^{\rho})\land ac(x-c_{\alpha0}^{1})\neq ac(c_{\alpha i}^{n}-c_{\alpha i}^{1})\right)
\end{eqnarray*}

and $d'_{\alpha i}=d_{\alpha i}\cup v(c_{\alpha i}^{n}-c_{\alpha i}^{1})\cup ac(c_{\alpha0}^{n}-c_{\alpha0}^{1})$.

~

In any case, we have that $\left\{ \phi_{\alpha}'(x,d_{\alpha i}')\right\} _{i<\omega}$
is inconsistent, $\left\{ \phi_{\beta}(x,d_{\beta,0})\right\} _{\beta<\alpha}\cup\left\{ \phi'_{\alpha}(x,d_{\alpha0}')\right\} \cup\left\{ \phi_{\beta}(x,d_{\beta0})\right\} _{\alpha<\beta<\delta}$
is consistent, and $\left(\bar{d}_{\beta}\right)_{\beta<\alpha}\cup\left\{ \bar{d}_{\alpha}'\right\} \cup\left(\bar{d}_{\beta}\right)_{\alpha<\beta<\delta}$
is a mutually indiscernible array. Doing this for all $\alpha$ by
induction we get an $\inp$-pattern of the same depth involving strictly
less different $v(x-y_{i})$'s --- contradicting the inductive hypothesis.
\end{proof}
Finally, we are ready to prove Theorem \ref{thm: Ax-Kochen}.
\begin{proof}
By the cell decomposition of Pas \cite{MR1004136}, every formula
$\phi(x,\bar{c})$ is equivalent to one of the form $\bigvee_{i<n}(\chi_{i}(x)\land\rho_{i}(x))$
where $\chi_{i}=\bigwedge\chi_{j}^{i}(v(x-c_{j}^{i}),\bar{d}_{j}^{i})$
with $\chi_{j}^{i}(x,\bar{d}_{j}^{i})\in L(\Gamma)$ and $\rho_{i}=\bigwedge\rho_{j}^{i}(ac(x-c_{j}^{i}),\bar{e}_{j}^{i})$
with $\rho_{j}^{i}(x,\bar{e}_{j}^{i})\in L(k)$. By Lemma \ref{lem: boolean operations on inp-patterns},
if there is an $\inp$-pattern of depth $\kappa$ with $x$ ranging
over $K$, then there has to be an $\inp$-pattern of depth $\kappa$
and of the form as in Lemma \ref{lem: linear formulas are NTP2},
which is impossible. It is sufficient, as $\Gamma$ and $k$ are stably
embedded with no new induced structure and are fully orthogonal.\end{proof}
\begin{problem}
~
\begin{enumerate}
\item Can the bound on $\kappa_{\inp}^{1}(K)$ given in Theorem \ref{thm: Ax-Kochen}
be improved? Specifically, is it true that $\kappa_{\inp}^{1}\left(K\right)\leq\kappa_{\inp}^{1}\left(k\right)\times\kappa_{\inp}^{1}\left(\Gamma\right)$
in the ring language?
\item Determine the burden of $K=\prod_{p\mbox{ prime}}\mathbb{Q}_{p}/\mathfrak{U}$
in the pure field language. In \cite{DolichGoodrickLippel} it is
shown that each of $\mathbb{Q}_{p}$ is $\mbox{dp}$-minimal, so combined
with Fact \ref{fac: In NIP, burden =00003D dp-rank} it has burden
$1$. Note that $K$ is not $\mbox{inp}$-minimal in the Denef-Pas
language, as the residue field is infinite, so $\{v(x)=v_{i}\}$,
$\{ac(x)=a_{i}\}$ shows that the burden is at least $2$. However,
Hrushovski pointed out to me that the angular component is not definable
in the pure ring langauge, thus the conjecture is that every ultraproduct
of $p$-adics is of burden $1$ in the pure ring (or $RV$) language.
\end{enumerate}
\end{problem}
\bibliographystyle{alpha}
\bibliography{common}

\end{document}